\documentclass[journal]{IEEEtran}
\usepackage{amsmath, amssymb, bbm}
\usepackage{amsthm}
\usepackage{graphicx,epsfig}
\usepackage{subfigure}
\usepackage{verbatim}
\usepackage{soul}
\usepackage[colorlinks]{hyperref}
\usepackage{color}

\newtheorem{theorem}{Theorem}
\newtheorem{lemma}{Lemma}
\newtheorem{corollary}{Corollary}
\newtheorem*{proof*}{Proof}

\usepackage{hhline}
\usepackage{multirow}
\usepackage{caption}

\usepackage{algorithm,algorithmic}

\usepackage{cite}

\newcommand{\jc}{\mathbf{j}}
\newcommand{\conj}[1]{\overline{#1}}
\newcommand{\br}[1]{\left({#1}\right)}
\newcommand{\Com}{\mathbb{C}}
\newcommand{\Vbase}{V^0}
\newcommand{\Vinv}{\gamma}
\newcommand{\diag}[1]{[\![ \, {#1} \, ]\!]}
\newcommand{\One}{\mathbf{1}}

\newcommand{\norm}[1]{\left\|{#1}\right\|}
\newcommand{\inv}[1]{{{#1}}^{-1}}
\newcommand{\s}{\star}

\begin{document}

\bstctlcite{IEEEexample:BSTcontrol}

\title{A Framework for Robust Long-Term Voltage Stability of Distribution Systems}
\author{Hung D. Nguyen\textsuperscript{*}, Krishnamurthy Dvijotham\textsuperscript{*}, Suhyoun Yu, and~Konstantin~Turitsyn

\thanks{Krishnamurthy Dvijotham is with the Department of Mathematics, Washington State University and with the Optimization and Control group, Pacific Northwest National Laboratory, Richland, WA 99354, USA email: dvij@cs.washington.edu.

Hung D. Nguyen, Suhyoun Yu, and Konstantin Turitsyn are with the Department of Mechanical Engineering, Massachusetts Institute of Technology, Cambridge, MA, 02139 USA e-mail: hunghtd@mit.edu, syu2@mit.edu, and turitsyn@mit.edu.

\textsuperscript{*} The first two authors contributed equally to this work.}}

\markboth{IEEE Transactions on Smart Grids,~Vol.~, No.~, September~2018}%
 {Shell \MakeLowercase{\textit{et al.}}}%

\maketitle

\begin{abstract}
Power injection uncertainties in distribution power grids, which are mostly induced by aggressive introduction of intermittent renewable sources, may drive the system away from normal operating regimes and potentially lead to the loss of long-term voltage stability (LTVS). Naturally, there is an ever increasing need for a tool for assessing the LTVS of a distribution system. This paper presents a fast and reliable tool for constructing \emph{inner approximations} of LTVS regions in multidimensional injection space such that every point in our constructed region is guaranteed to be solvable. Numerical simulations demonstrate that our approach outperforms all existing inner approximation methods in most cases. Furthermore, the constructed regions are shown to cover substantial fractions of the true voltage stability region. The paper will later discuss a number of important applications of the proposed technique, including fast screening for viable injection changes, constructing an effective solvability index and rigorously certified loadability limits. 
\end{abstract}


\section{Introduction}
The long-term voltage stability (LTVS) is an important class of system voltage stability that studies the impact of slow dynamic components such as tap-changers and thermal loads in several-minute time scales \cite{kundur2004definition}. A system can become unstable following a number of different possible mechanisms, such as the loss of long-term equilibrium. In particular, power injection uncertainties, often induced by distributed renewable energy resources and possibly demand response programs, may push the system out of its viable region, where its steady-state equilibrium ceases to exist \cite{overbye1994power, overbye1995computation}.

Accordingly, the solution of power flow problem, which corresponds to the steady-state equilibrium, plays a critical role in LTVS assessment. The disappearance of solutions implies that the given injection, or the operating point, is beyond the network's solvability limit. That is, the network is incapable of supporting the amount of demanded load. For this reason, it is crucial for a network's LTVS that power system operators are aware of allowed ranges, or security regions, wherein the injections may vary without jeopardizing system stability. 

The LTVS of power systems has been studied for several decades. A review of literature in this problem can be found in  \cite{loparo1990probabilistic, WuSteady, bolognani2016existence, wang2016existence, aolaritei2017distributed}. Quite a few numerical techniques have been proposed to compute the critical loading level associated with LTVS, for example, the continuation power flow (CPF) \cite{ajjarapu1992continuation, Ajjarapu2005fastHopf}, and singular Jacobian-based methods introduced in \cite{lof1993voltage, ali2017transversality}. Apart from these numerical methods, analytic approaches that rely on explicit solvability certificates can provide an alternative solution for constructing the solvability region in all dimensions. The first region-wise certificates were introduced in \cite{Schweppe1975, Zhao2015,ilic1986localized, saric2008applications, makarov2000computation}. Several advantages of these certificates have been discussed in \cite{loparo1990probabilistic}, including the ability to provide security measures and the possibility of reducing computational costs while considering multiple loading directions. The latter advantage follows from the fact that, unlike numerical approaches, the certificates can be reused even after the directions change.

Unfortunately, most region-wise approaches suffer from conservatism issues in which the characterized sets can become overly small. In recent work, Banach's fixed-point theorem has been successfully applied to distribution systems and shown to construct large subsets of the stability region \cite{bolognani2016existence, EssieHungKostya, wang2016existence}. Among these, the results presented in \cite{wang2016existence} (denoted WBBP in our paper based on the authors' last names) dominate all previous results. However, WBBP's solvability criterion requires a specific condition on the nominal point, around which the solvability region approximation is constructed. In the regime where this condition is close to being violated, the estimated regions become conservative as shown in section \ref{sec:largesys}.

In this work, we propose to use Brouwer's fixed-point theorem---a popular fixed-point theorem (especially in market economics \cite{gale1979game})---to overcome the conservative nature of the aforementioned methods. The constructed regions in the parameter space have a simple analytical form, i.e., a norm constraint on an affine function of the power injection inputs. This leads to significant computational advantages, and we outline several power systems applications which can employ the regions we construct. We demonstrate tightness of the solvability regions in the sense that our regions almost ``touch'' the boundary of the true, usually nonconvex, feasibility set. As a side note, while \cite{simpson2017theory2}, \cite{simpson2017theory1} have applied Brouwer's fixed-point theorem to decoupled power flow problems in lossless radial networks \cite{simpson2017theory1, simpson2017theory2}, our approach considers full AC power flow equations which describe the actual lossy systems.

In the scope of this paper, we only consider distribution systems with constant power $PQ$ buses, and we neglect shunt elements, tap changers, switching capacitor banks and other discrete controls on the grid. Though it is possible to incorporate a simple bound on the voltage magnitudes, we do not aim to enforce operational constraints. Finally, the topology of the system is assumed to be unchanged; we do not consider contingencies involving transmission line losses. Some extensions mentioned above are addressed in our more recent work \cite{hung2017inner}.




The main contributions of the paper are as follows. In section \ref{sec:background}, we introduce the set of notations used in the paper, then we derive power flow equations and introduce Brouwer's fixed-point theorem. In section \ref{sec:solcriterion}, we present our main technical contribution, in particular, the sufficient condition for solvability of AC power flow equations that appear in a single-phase distribution system (not necessarily radial) with PQ buses only. This solvability condition allows one to construct inner approximations of the solvability region in multidimensional injection space. Section \ref{sec:application} discusses a number of practical applications of our result: a fast screening technique for speeding up LTVS analysis; the development of a fast-to-compute and informative index for solvability; rigorous techniques for computing certified loading gain limits. Finally, in section \ref{sec:Numeric}, we illustrate our technique's performance by testing on multiple IEEE distribution test feeders using MATPOWER. The simulation results show that our estimated regions can cover up to $80\%$ of the true solvability, thus being sufficiently large for operational purposes, and in most of cases, our approximation dominates WBBP's results.

\section{Background and notation} \label{sec:background}
The following notations will be used throughout this paper:
\begin{align*}
& \Com: \text{ Set of complex numbers} \\
	& [\![ \, \pmb{x} \, ]\!] = \text{diag}(\pmb{x}) \text{ for } \pmb{x}\in\mathbb{C}^n,\conj{\pmb{x}}: \text{ Conjugate of } \pmb{x}\in\mathbb{C}^n \\
    & \One: \text{Vector of compatible size with all entries equal to 1} \\
    & \mathbf{I}: \text{Identity matrix of compatible size}
\end{align*}
\begin{align*}
    & \|\pmb{x}\| = \|\pmb{x}\|_{\infty} = \max_i{x_i}\text{ for } \pmb{x}\in\mathbb{C}^n \\
    & \|\pmb{A}\| = \|\pmb{A}\|_{\infty} = \max_i\sum_j|A_{ij}|\text{ for } \pmb{A}\in\mathbb{C}^{n \times n} \\
    & \frac{\partial F}{\partial x} = \begin{pmatrix} \frac{\partial F_1}{\partial x_1} & \ldots & \frac{\partial F_n}{\partial x_1} \\ \vdots & \vdots & \vdots \\ \frac{\partial F_n}{\partial x_1} & \ldots & \frac{\partial F_n}{\partial x_n}\end{pmatrix} \text{ for } F:\Com^n \mapsto \Com^n
\end{align*}

We study power grids composed of one slack bus and PQ buses. We use $0$ to denote the slack bus and $1,\ldots,n$ to denote the PQ buses. The slack bus voltage $V_0$ will be fixed as a reference value, and $V_1,\ldots,V_n$ are variables. The complex net power injection at bus $i$ will be denoted as $S_i=P_i+\jc Q_i$, where $P_i$ is the active power injection, and $Q_i$ is the reactive power injection. The sub-matrix of the admittance matrix corresponding to the PQ buses can be constructed by eliminating the first row and column, and this sub admittance matrix will be denoted by $\mathbf{Y} \in \Com^{n \times n}$ and its $\br{i,k}$-th entry as $Y_{ik}$. The power balance equations can then be written as
\begin{align}
\conj{Y_{i0}}V_i\conj{V_0}+\sum_{k=1}^n \conj{Y_{ik}}V_i\conj{V_k}=S_i=P_i+\jc Q_i, \quad i = 1,\dots,n.
\end{align}

\noindent Let $\Vbase$ denote the voltage solution associated with the zero injection condition, which corresponds to the zero current state in the absence of shunt elements. Then, $\Vbase$ is the solution to the system of linear equations
\begin{align}
\sum_{k=1}^n \conj{Y_{ik}}\Vbase_i+\conj{Y_{i0}}V^0=0, \quad i = 1,\dots,n.
\end{align}
The power-flow equations can then be expressed in the following compact form:
\begin{equation}\label{eq:PFbasic}
\diag{\pmb{V}}\overline{\pmb{Y}\br{\pmb{V}-\pmb{\Vbase}}} = \pmb{S}  \end{equation}
where $\pmb{\Vbase} = V^0\, \One$ is the voltage vector with all entries are $V^0$. Note that the admittance matrix $\pmb{Y}$ in \eqref{eq:PFbasic} is not the full matrix constructed for all buses of a network, but rather the sub-matrix obtained after removing the slack bus. This sub-matrix is non-singular, thus being invertible \cite{Bergen, wang2016existence}.

\begin{theorem}[Brouwer's fixed-point theorem \cite{spanier1994algebraic}]
Let $F: \mathcal{U} \mapsto \mathcal{U}$ be a continuous map, where $\mathcal{U}$ is a compact and convex set in $\mathbb{C}^n$. Then the map has a fixed-point in $\mathcal{U}$. That is, $F\br{\pmb{x}}=\pmb{x}$ has a solution in $\mathcal{U}$.
\end{theorem}

The Brouwer's fixed-point theorem (see Chapter 4, Corollary 8 in \cite{spanier1994algebraic}) has several forms, among which this particular form of the theorem applies to compact and convex sets in Euclidean space. Apart from the assumption on the continuity of the map, the theorem also requires the self-mapping condition; in particular, the domain and the codomain are the same sets. In section \ref{sec:solcriterion}, we will use this theorem to derive sufficient conditions on $\pmb{S}$ so that the conditions guarantee the existence of solutions of the power flow equations. 

In particular, we will implement the following steps. First, we transform the traditional power flow equations into a fixed-point form of voltage variables (see \eqref{eq:AAFP}), so that Brouwer's fixed-point theorem can apply. The resulted fixed-point function will admit the power $\pmb{S}$ as the parameters. Next, we define the set $\mathcal{U}$ as a ball in the voltage space (see the proof of Theorem \ref{cert}). Then we characterize an admissible set of the parameter $\pmb{S}$ which results in a self-mapping function within the ball $\mathcal{U}$. The self-mapping condition is imposed by confining the image of the map within the domain. Then, the admissible set of $\pmb{S}$ will ensure the existence of a power flow solution following Brouwer's fixed-point theorem. Note that there are similar results for more general sets of quadratic equations have been reported in our recent work \cite{dvijotham2017solvability}.

\section{Solvability certificates} \label{sec:solcriterion}
In this section, we will apply Brouwer's fixed-point theorem to the full AC power flow equations in \eqref{eq:PFbasic}. The central result for the existence of a steady-state solution is introduced below.
\begin{theorem}\label{cert}
Let $\pmb{V}_\s$ be a solution to the power flow equations \eqref{eq:PFbasic}. Define
\begin{subequations}
\begin{align}
& \pmb{Z}_\s= \inv{\diag{\conj{\pmb{V}_\s}}}\inv{\conj{\pmb{Y}}}\inv{\diag{\pmb{V}_\s}} \label{eq:Zs}\\
& \pmb{J}_\s=	\begin{pmatrix}
\mathbf{I} &  \conj{\pmb{Z}_\s}\diag{\conj{\pmb{S}_\s}} \\
 \vspace{.01em} &  \vspace{.01em} \\
\pmb{Z}_\s \diag{\pmb{S}_\s} & \mathbf{I}
\end{pmatrix} \\
& \br{\pmb{J}_\s}^{-1}=\begin{pmatrix}
\pmb{M}_\s &  \pmb{N}_\s\\
\vspace{.01em} & \vspace{.01em} \\
 \conj{\pmb{N}_\s} & \conj{\pmb{M}_\s} \end{pmatrix}. \end{align}	
\end{subequations}

Let $\pmb{S} \in \Com^n,r>0$ be arbitrary and define $\Delta \pmb{S}=\pmb{S}-\pmb{S}_\s$. Then, \eqref{eq:PFbasic} has a solution if
\begin{align}
&\frac{1}{r}\norm{\pmb{M}_\s\conj{\pmb{Z}}_\s\conj{\Delta \pmb{S}}+\pmb{N}_\s \pmb{Z}_\s \br{\Delta \pmb{S}}} +\norm{\inv{\pmb{J}_\s}}\norm{\pmb{Z}_\s \diag{\pmb{S}}}r \nonumber \\
& +\norm{\pmb{M}_\s\conj{\pmb{Z}_\s} \diag{\conj{\Delta \pmb{S}}}+\pmb{N}_\s\diag{\pmb{Z}_\s\br{\Delta \pmb{S}}}}\nonumber \\
& +\norm{\pmb{M}_\s\diag{\conj{\pmb{Z}_\s\br{\Delta \pmb{S}}}}+\pmb{N}_\s \pmb{Z}_\s\diag{\Delta \pmb{S}}}\leq 1. \label{eq:Condition}
\end{align}
Further, if $r<1$, the solution $V$ lies in the set
\begin{align}\frac{|V_{\s i}|}{1+r} \leq |V_i|\leq \frac{|V_{\s i}|}{1-r}.	 \label{eq:CondVolt}
\end{align}
\end{theorem}

\begin{proof}[Proof (Theorem \ref{cert})]
Define
\begin{equation}
\zeta\br{\pmb{S}}=\conj{\pmb{Z}_\s}\diag{\conj{\pmb{S}}},\eta\br{\pmb{S}}=\conj{\pmb{Z}_\s \pmb{S}}   
\end{equation}

Using lemma \ref{lem:PFrewrite} in the Appendix, we can rewrite \eqref{eq:PFbasic} as
\begin{align}
\pmb{y}+\zeta\br{\pmb{S}_\s} \conj{\pmb{y}} & =-\eta\br{\Delta \pmb{S}}-\diag{\eta\br{\Delta \pmb{S}}}\pmb{y}-\zeta\br{\Delta \pmb{S}}\conj{\pmb{y}} \nonumber\\
& \quad -\diag{\pmb{y}}\zeta\br{\pmb{S}} \conj{\pmb{y}}	\label{eq:AA}
\end{align}
where $\pmb{y}=[\![\pmb{V}]\!]^{-1}\pmb{V}_\s-\One$, and $[\![\pmb{V}]\!]^{-1}\pmb{V}_\s$ is the component-wise division of  $\pmb{V}_{\s}$ and $\pmb{V}$. Let $\pmb{\alpha}$ denote the {RHS of (5)}. We then have
\begin{equation}
    \begin{pmatrix} \pmb{\alpha} \\ \conj{\pmb{\alpha}}\end{pmatrix} = \begin{pmatrix} \mathbf{I} & \conj{\pmb{Z}_\s}\diag{\conj{\pmb{S}_\s}} \\ \pmb{Z}_\s \diag{\pmb{S}_\s} & \mathbf{I} \end{pmatrix}\begin{pmatrix} \pmb{y} \\ \conj{\pmb{y}} \end{pmatrix}
\end{equation}
or
\begin{equation}
   \inv{\pmb{J}_\s} \begin{pmatrix} \pmb{\alpha} \\ \conj{\pmb{\alpha}}\end{pmatrix} =  \begin{pmatrix} \pmb{y} \\ \conj{\pmb{y}} \end{pmatrix} 
\end{equation}
{Solving for $\pmb{y}$ from the equation right above, we can see that
\begin{equation}
\pmb{y} = M_\s\pmb{\alpha} + N_\s\overline{\pmb{\alpha}},
\end{equation}
}
\noindent {so that after expanding the expression for $\pmb{\alpha}$, equation} \eqref{eq:AA} can be rewritten as
\begin{align}
\pmb{y}& =-\br{\pmb{M}_\s \eta\br{\Delta \pmb{S}}+\pmb{N}_\s\conj{\eta\br{\Delta \pmb{S}}}} \nonumber \\
& \qquad -\br{\pmb{M}_\s \diag{\pmb{y}}\zeta\br{\pmb{S}}\conj{\pmb{y}}+\pmb{N}_\s\diag{\conj{\pmb{y}}}\conj{\pmb{\zeta}\br{\pmb{S}}}\pmb{y}}\nonumber\\
& \qquad -\br{\pmb{M}_\s \diag{\eta\br{\Delta \pmb{S}}}+\pmb{N}_\s \conj{\zeta\br{\Delta \pmb{S}}}}\pmb{y} \nonumber \\
& \qquad - \br{\pmb{M}_\s \zeta\br{\Delta \pmb{S}}+\pmb{N}_\s\diag{\conj{\eta\br{\Delta \pmb{S}}}}}\conj{\pmb{y}} \label{eq:AAFP}
\end{align}

We apply Brouwer's fixed-point theorem to \eqref{eq:AAFP} with the set $\{\pmb{y}:\norm{\pmb{y}}\leq r\}$. We take the norm of the RHS of \eqref{eq:AAFP} and apply triangle inequality and the definition of the matrix norm to obtain:
\begin{align}
& \norm{\pmb{M}_\s \eta\br{\Delta \pmb{S}} + \pmb{N}_\s\conj{\eta\br{\Delta \pmb{S}}}} +\br{\norm{\pmb{M}_\s}+\norm{\pmb{N}_\s}}\norm{\zeta\br{\pmb{S}}}r^2 \nonumber\\
& +\norm{\pmb{M}_\s \zeta\br{\Delta \pmb{S}} + \pmb{N}_\s \diag{\conj{\eta\br{\Delta \pmb{S}}}}}r\nonumber \\
& +\norm{\pmb{M}_\s \diag{\eta\br{\Delta \pmb{S}}}+\pmb{N}_\s\conj{\zeta\br{\Delta \pmb{S}}}}r \label{eq:AAb}
\end{align}
Since \eqref{eq:AAb} is an upper bound on the norm of the RHS of \eqref{eq:AAFP}, Brouwer's fixed-point theorem guarantees that \eqref{eq:AAFP} has a solution if \eqref{eq:AAb} is smaller than $r$.
Dividing \eqref{eq:AAb} by $r$ and requiring the result to be smaller than $1$, we obtain \eqref{eq:Condition} which establishes the theorem.
Moreover, the solution will exist in the set $\norm{ [\![\pmb{V}]\!]^{-1}\pmb{V}_\s - 1}\leq r$, or
\begin{equation}
    \left |\frac{V_{\s i}}{V_i}-1\right | \leq r \implies  |V_{\s i}-V_i| \leq r |V_i|
\end{equation}
Applying the triangle inequality, we obtain $|V_{\s i}|-|V_i| \leq r |V_i|,|V_i|-|V_{\s i}| \leq r |V_i|$ or
\begin{equation} \label{eq:ballr}
    \frac{|V_{\s i}|}{1+r} \leq |V_i| \leq \frac{|V_{\s i}|}{1-r}
\end{equation}
\end{proof}

The sufficient condition \eqref{eq:Condition} defines a convex approximation of the solvability set. The convexity can be seen as the left hand side of \eqref{eq:Condition} is the sum of four terms of the type $\norm{\pmb{A}^T \Delta \pmb{S} + \pmb{b}}$. Each individual term is a convex function due to the triangle inequality. Then the convexity of the approximated sets follows from the fact that the sum of convex functions is convex.


A related point to consider is that, unlike the approach based on Banach's fixed-point theorem, the Brouwer approach developed in this work does not guarantee the uniqueness of the solution. However, it can provide other information about the solutions, namely the voltage range within which the solutions will lie (the set defined by \eqref{eq:CondVolt}). Knowing the solvable voltage range is indeed crucial to practical operation of power systems as the operators concern about not only whether steady-state equilibrium exists but also whether such solutions are compliant with the voltage requirements. Another advantage of the Brouwer approach is that its central condition--the self-mapping condition--can easily handle operational constraints such as current thermal and power generation limits. Thus the technique lends itself to feasibility constrained problems, for instance, to estimate the feasibility set of an Optimal Power Flow. More details are presented in our more recent work \cite{hung2017inner}.


As our primary focus, is to construct inner approximations of the solvability region around a base operating point, we necessarily assume the solvability of such base point. In normal situations, the current operating point is solvable and is a suitable base operating point. Otherwise, the zero power and zero current condition is a trivial base operating point. The solvability of the operating point, consequently, implies that the base Jacobian $\pmb{J}_\s$ is non-singular.

The solvability condition presented in \eqref{eq:Condition} is particularly useful if one is seeking for a solution that lies within some voltage bounds characterized by $r$. The value of $r$ reflects the size of the solvability region; hence if one's focus is on solvability of given injections, then one might be interested in finding the largest possible estimated subsets, which can be found by optimizing the value of $r$, thus directing to the following result.

\begin{corollary} \label{cor:union}
Let $\mathcal{U}_r$ denote the set of $s$ satisfying \eqref{eq:Condition}. Define the set
\begin{equation}
    {\mathcal{U}_r}=\cup_{0<\epsilon \leq 1} \mathcal{U}_{\epsilon r}.
\end{equation}
Then, the two following statements hold true.
\begin{itemize}
    \item For every $\pmb{S} \in {\mathcal{U}_r}$, there exists a solution $\pmb{V}$ to \eqref{eq:PFbasic} satisfying \eqref{eq:CondVolt}.
    \item Further, for every $\pmb{S}$ satisfying
    \begin{align}
    & 2\sqrt{\norm{\pmb{M}_\s\conj{\pmb{Z}}_\s\diag{\conj{\Delta \pmb{S}}}+\pmb{N}_\s \pmb{Z}_\s \diag{\Delta \pmb{S}}}\norm{\inv{\pmb{J}_\s}}\norm{\pmb{Z}_\s \diag{\pmb{S}}}} \nonumber \\
    & +\norm{\pmb{M}_\s\conj{\pmb{Z}_\s} \diag{\conj{\Delta \pmb{S}}}+\pmb{N}_\s\diag{\pmb{Z}_\s\br{\Delta \pmb{S}}}} \nonumber \\
    & +\norm{\pmb{M}_\s\diag{\conj{\pmb{Z}_\s\br{\Delta \pmb{S}}}}+\pmb{N}_\s \pmb{Z}_\s\diag{\Delta \pmb{S}}}\leq 1 \label{eq:ConditionInf}
    \end{align}
\end{itemize}
there is a solution $\pmb{V}$ to \eqref{eq:PFbasic}. \end{corollary}

Both claims of Corollary \ref{cor:union} can be proven following Theorem \ref{cert} by showing that there exists an equivalent certificate in the form of \eqref{eq:Condition} characterized by an associated value of radius $r$. For the first claim, each injection $\pmb{S}$, which lies in $\mathcal{U}_r$, will also belong to some subset $\mathcal{U}_{\epsilon r}$ contained in $\mathcal{U}_r$. Thus, the associated radius is simply $\epsilon r$. The second claim is proven by choosing a specific value of $r$ given by
\begin{equation}
    r = \sqrt{\frac{\norm{\pmb{M}_\s\conj{\pmb{Z}}_\s\diag{\conj{\Delta \pmb{S}}}+\pmb{N}_\s \pmb{Z}_\s \diag{\Delta \pmb{S}}}}{\norm{\inv{\pmb{J}_\s}}\norm{\pmb{Z}_\s \diag{\pmb{S}}}}}
\end{equation}
which transforms \eqref{eq:Condition} into \eqref{eq:ConditionInf}.

Furthermore, the fact that $\mathcal{U}_r$ is the union of all $\mathcal{U}_{r'}$ where $0 \leq r' \leq r$ implies that, for more restricted range of voltage solutions, the corresponding subset $\mathcal{U}_{r'}$ will be contained in the subset $\mathcal{U}_{r}$. This result reflects the fact that more varying power injections will cause a wider range of voltage variation. 

The proposed solvability conditions \eqref{eq:ConditionInf} and \eqref{eq:Condition} share the same physical interpretation which reflects the behavior of power systems. Both conditions are norm constraints on the parameter variation $\Delta \pmb{S}$ representing loading conditions. As $\Delta \pmb{S}$ increases, the norm of $\Delta \pmb{S}$ increases as well and will eventually violate the solvability conditions. Since our conditions are sufficient but not necessary conditions for the existence of a solution, $\Delta \pmb{S}$ may fail to satisfy our condition before the true point of voltage collapse. Therefore, the constructed region is conservative. The numerical evaluation in section \ref{sec:largesys} studies this issue by looking at the ratio between the maximum loading for which $\Delta \pmb{S}$ satisfies condition \eqref{eq:ConditionInf}. Moreover, these conditions also involve the inverse of the Jacobian at the base point. As the base point moves toward the solvability boundary, the Jacobian becomes close to singular and the norm of its inverse increases. In short, a slight increase in the loading level may violate the sufficient solvability criterion, implying that the solvability margin tends to decrease as the base point gets closer to solvability boundary.

To gain insight into the region characterized by the proposed certificates, we consider two special cases: the first is when the base operating condition has the zero power and the zero current in Appendix \ref{app:nullinject}, and the second is when the estimated bound almost touches the real one under the coalescence condition presented in section \ref{sec:application}

In comparison to WBBP's estimate introduced in \cite{wang2016existence}, one of the least conservative known approximations, our constructions outperform in most of the studied cases. More specifically, we can provide a rigorous proof for the dominance of our estimations over WBBP’s results for the zero loading condition (see Appendix \ref{app:nullinject} and also \cite{dvijotham2017solvability}). For nonzero power operating points, simulation results presented in section \ref{sec:Numeric} also show that our estimates are most often superior. We observed a few exceptional cases wherein WBBP's estimations contain ours (\ref{fig:sol123bus}(b)), but we do not aim to provide any mathematical comparison between the two approaches for general base operating conditions. Another advantage of our framework is that we can obtain a region around any nominal solution $\br{\gamma^\star,s^\star}$ with non-singular Jacobian, while WBBP's approach requires a stronger assumption. Otherwise, it should be noted that our comparison is not entirely fair to WBBP, as the WBBP's condition guarantees uniqueness, whereas our condition only certifies existence.


\section{Applications} \label{sec:application}
The proposed sufficient solvability criteria are useful for a number of important operational functions including, but not limited to, verifying viable injections, loadability limit monitoring, and security-constrained optimization functions. As discussed at the end of section \ref{sec:solcriterion}, it is possible to construct approximated convex regions which each has a simple analytic form. Such convex shapes then can be incorporated into the constrained optimization by replacing the original nonlinear power flow equations. Nevertheless, in the scope of this paper, we only consider three immediate applications: fast screening for viable injection change, effective solvability index, and certified loadability limit estimation.

\subsection{Fast screening for viable injection change} \label{sec:fastscr}

The verification problem mainly concerns whether an injection is viable or not. Carrying out the verification over the real solvability region is challenging because the actual boundary is difficult to construct. Alternatively, we propose to use the approximated region characterized by sufficient solvability criteria. If the approximated solvability region is convex, the verification problem is indeed a membership oracle, a basic algorithmic convex geometry problem \cite{vempala2010recent}. Once an injection is verified, the corresponding operating point is guaranteed to be solvable. Otherwise, other detailed tests need to take place. We introduce an algorithm for the purpose of fast screening below.


The screening problem usually considers a cloud of points in the injection space, or the potential injection set, and the task is as follows:

\textbf{Screening problem:} Given a set of potential injections $\mathcal{P}$, classify all solvable and unsolvable points to sets $\mathcal{F}$ and $\mathcal{I}$, respectively.

An injection is solvable if it belongs to a subset defined by criterion \eqref{eq:ConditionInf}. In our approach, we construct multiple solvability subsets to screen all solvable injection points. As one subset can be reused to verify multiple points, the screening time can be reduced significantly.

A fast screening procedure is presented in Algorithm \ref{alg:fastscr}. In the proposed algorithm, power flow (PF) (or continuation power flow (CPF) if PF does not converge) only needs to perform for the candidate scenario, or seed points. If any point in the given potential injection set satisfies the solvability condition \eqref{eq:ConditionInf} associated with the seed point, it is certified as solvable. More specifically, one needs to verify the condition \eqref{eq:ConditionInf} while assigning the seed point power level as $\pmb{S}_\s$ and the potential injection level as $\pmb{S}$. Among uncertified points which may or may not be solvable, we select another seed point and continue the screening process until all points from the potential injection set are classified.

\begin{algorithm}[t]
\caption{Fast screening algorithm based on Brouwer's theorem}
\label{alg:fastscr}
\begin{algorithmic}[1]
\STATE {Store all potential injections in a set $\mathcal{P}$}
\STATE {Initialize a set $\mathcal{I}$ and $\mathcal{F}$ as an empty set}
\STATE {\textbf{While} $\mathcal{P}$ is not empty \textbf{do}}

\begin{itemize}
\item Choose the first point as a seed\\
\item Solve PF (or CPF) for the seed and remove the seed from $\mathcal{P}$\\
	{\hspace{5pt} {\textbf{if} solvable  \textbf{then}}}\\
    {\hspace{15pt} add the seed to $\mathcal{F}$}\\
    	{\hspace{15pt} \textbf{for} $i = 1,\cdots,\textrm{card}(\mathcal{P})$ \textbf{do}}\\
    		{\hspace{25pt} \textbf{if} $\mathcal{P}(i)$ satisfies \eqref{eq:ConditionInf} w.r.t the seed \textbf{then}}\\
                     {\hspace{30pt} remove $\mathcal{P}(i)$ from $\mathcal{P}$ and add it to $\mathbf{F}$}\\
            {\hspace{25pt} \textbf{end if}}\\
        {\hspace{15pt} \textbf{end for}}\\
     {\hspace{5pt} {\textbf{else}}}\\
   {\hspace{15pt} add the seed to $\mathcal{I}$}\\
   {\hspace{5pt} \textbf{end if}}
\end{itemize}

\STATE {Return $\mathcal{I}$ and $\mathcal{F}$}
\end{algorithmic}
\end{algorithm}

In contrast, without a fast screening procedure, one typically solves the PF or CPF for each scenario in the set $P$ to determine its solvability. Figure \ref{fig:fastscr} illustrates the performance of the fast screening algorithm against a CPF-based one in terms of elapsed time. In this simulation, the potential injection set is generated uniformly randomly which mimics uncertain renewable injections. As expected, the time acceleration factor---{a multiplier between two processing times consumed in the fast and CPF-based screening procedures}---tends to increase when more scenarios are considered. The simulation results show that the proposed method can speed up the screening up to $400$ times. According to our numerical observation, the performance of the fast screening method depends on the density of the potential injection set. The more concentrated the potential injection points are, the faster the screening outperforms.
\begin{figure}[!ht]
    \centering
    \includegraphics[width=1.\columnwidth]{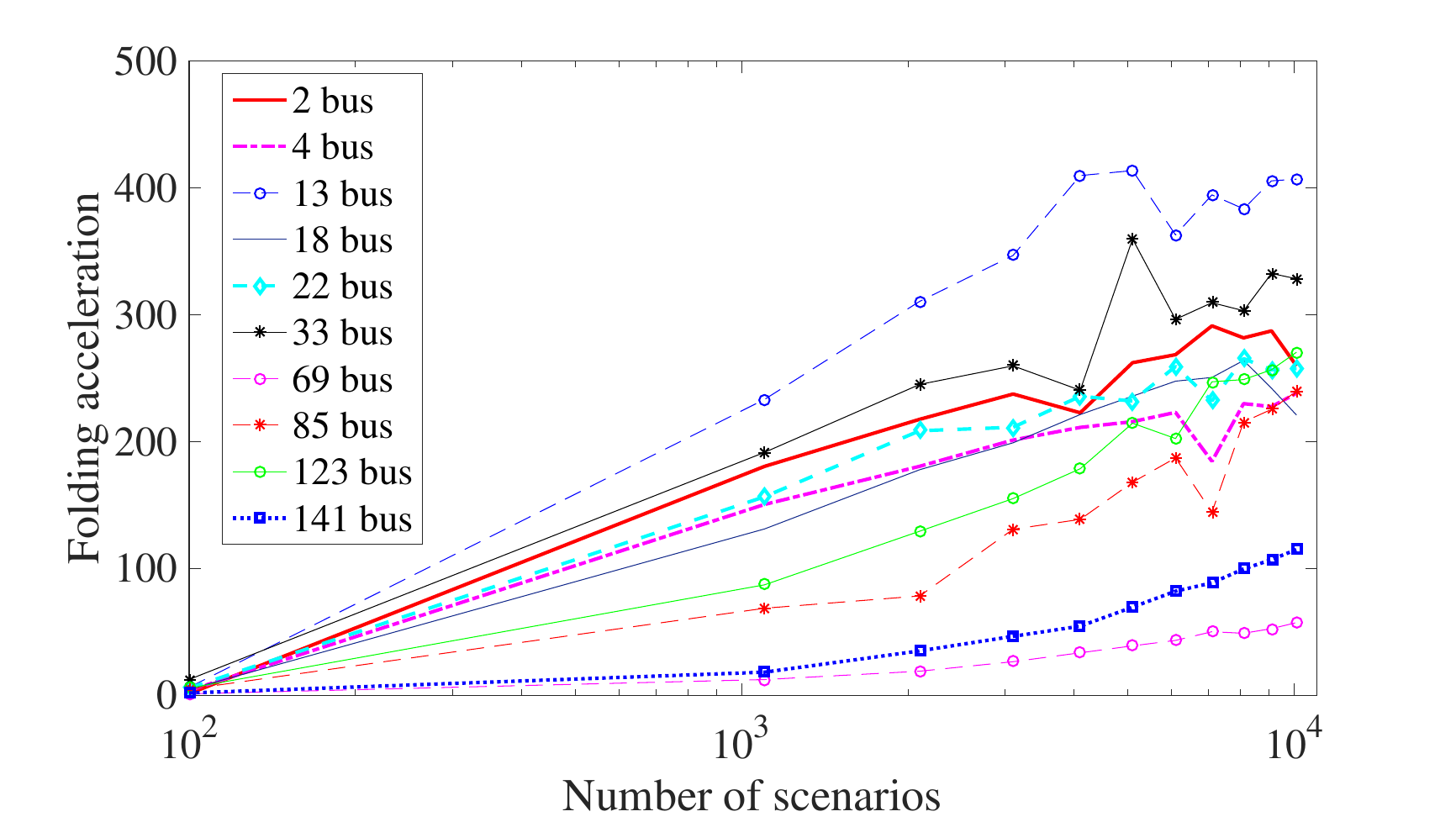}
	\caption{The performance of the fast screening algorithm against a CPF-based approach}
     \label{fig:fastscr}
\end{figure}

Moreover, the performance of the proposed fast screening technique depends on the relative sizes of the feasible and infeasible sets, simply because the infeasible points cannot be verified by our certificates. Many infeasible points consequently cause the proposed approach to take more computation time compared to the CPF-based method as we have to construct our certificates. The folding acceleration result presented in Figure \ref{fig:fastscr} corresponds to a small set of infeasible points that make up less than $5\%$ of the potential injection points. The fast screening method is primarily proposed for monitoring purposes where one needs to assess whether the system will be ``safe'' in the next a few minutes. Within this short period, excessive predicted infeasible injections usually indicate that the system will likely exhibit voltage stability problems.

\subsection{Effective solvability index} \label{sec:probmeas}
The above fast screening needs to verify all potential injection points by constructing multiple certificates; the proceeding section focus on the following problem.

\textbf{Effective solvability index calculation:} Given a set of potential injections $\mathcal{P}$, compute the percentage of injection points which can be verified with a single certificate.

This problem is to demonstrate the effectiveness of a single certificate. To solve the problem, we first construct a solvability subset based on condition \eqref{eq:ConditionInf} then calculate the percentage of points that lie inside the constructed region. This percentage can serve as an estimated measure of solvability---or as we denote it as the effective solvability index---to help the system operators to quickly make decisions on whether or not to continue operating the system (with the same settings) under a given level of uncertainty of power injections. An example of the procedure calculating the index is described below.
\begin{figure}[ht]
    \centering
    \includegraphics[width=1.\columnwidth]{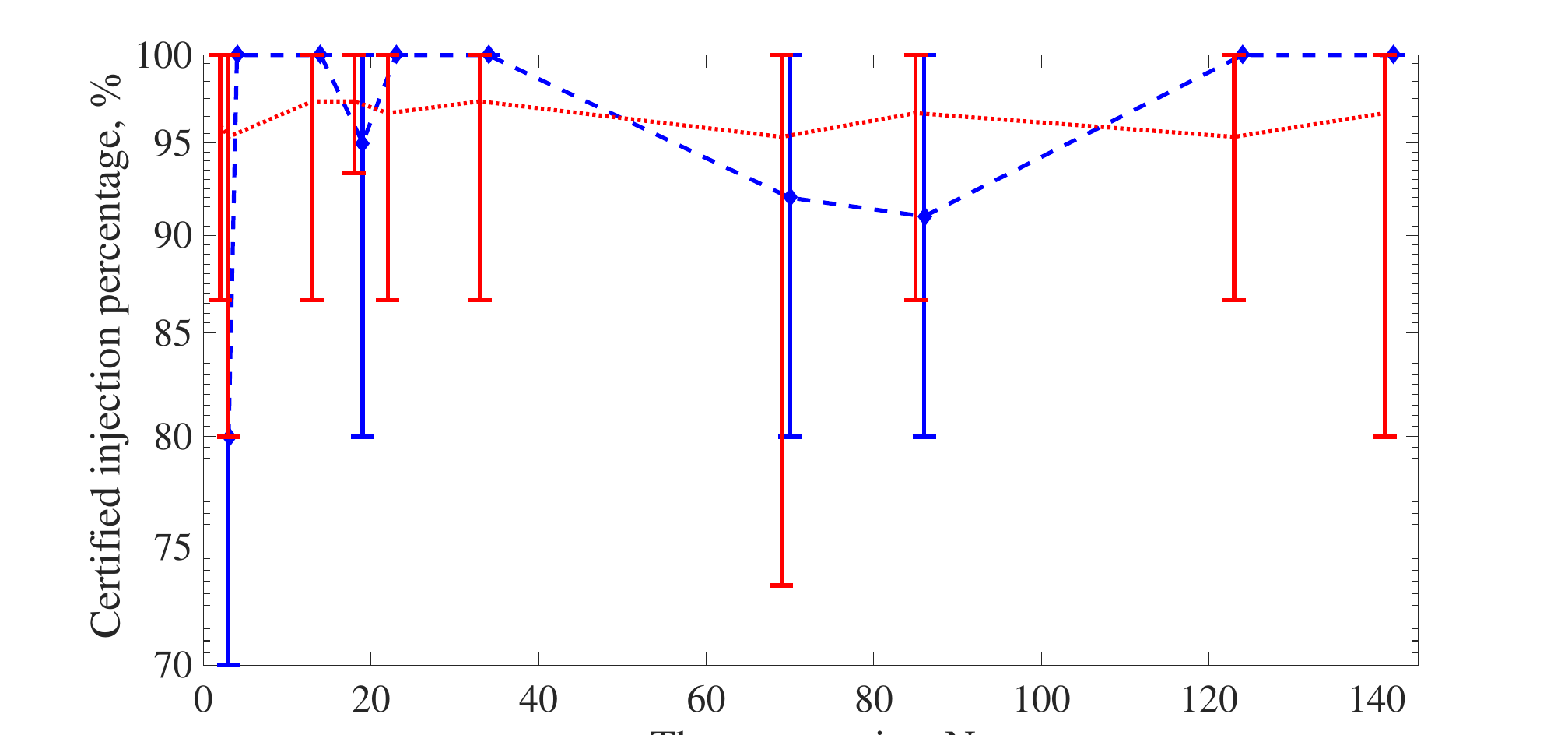}
	\caption{The effective solvability index}
		\label{fig:renewable}
\end{figure}

Figure \ref{fig:renewable} illustrates the performance, in terms of the percentage of certified random injections, of a single certificate characterized by \eqref{eq:ConditionInf} for several IEEE distribution test feeders with renewables. We modify the test cases to accommodate $\sim$35$\%$ renewable penetration by installing photovoltaic panels (PVs) at more than one third of the load buses. The PV locations are selected uniformly and randomly. Potential injection sets are generated by varying the PV outputs with two separate methods; the first method randomly selects PV outputs from a uniform distribution over a range of $\pm 500\%$ of the base load/injection, while the second method randomly and uniformly selects from the forecast data for May 03, 2017 available online \cite{PVforecast}. In Figure \ref{fig:renewable}, the results were recorded in the percentage of number of injections certified, where the red and blue data corresponds to the first and second PV output selection methods, respectively.

As indicated in the plot, one single certificate, on average, is able to certify $\sim95\%$ of the random injection sets, and even the lower bound of this percentage is well over the majority. For all test cases, it is also evident that the certificate may even extend to certify exhaustively the potential sets. The simulation thus illustrates that, even with the high uncertainty and randomness added from the renewable injections, one certificate can encompass a considerable portion of the solvability region. Apparently, though the performance of one certificate might depend on the test case configurations as well as the random injection sets, the operators can carry out this quick test to roughly estimate the system viability for a given amount of uncertain injections. If the index above an acceptable level of security, say $0.95$, no more assessment is required.

\subsection{Certified admissible gain limits}
In practice, the system operators may be interested in the distance between the current operating point and the insolvable boundaries. For each possible loading direction, one can normalize the corresponding incremental loading vector $\Delta \pmb{S}$, and then solve for the maximum gain, $\lambda_{max}$, for which the system remains solvable. The smallest $\lambda_{max}$ of all possible loading directions, or loading gain limit, can also be used to quantify the system stability. We can define the associated problem formally as the following.

\textbf{Loading gain limit problem}: For a given base injection power $\pmb{S}_\s$, find the maximum gain $\lambda_{max}$ for which the system will remain solvable along all possible normalized loading direction $\Delta \pmb{S}$, or $\norm{\Delta \pmb{S}} = 1$.

Solving for the loading gain limit can be formulated as a $\min-\max$ optimization problem described below:
\begin{align} \label{eq:certgain}
&\underset{\Delta{u}}{\text{min}}
\quad \underset{\lambda}{\text{max}}\quad \lambda \\
& \text{subject to} \quad \pmb{S}_\s + \lambda \Delta{\pmb{S}} \,\, \text{is solvable}, \nonumber\\
& \qquad \qquad \quad 0 \leq \lambda,\, \|\Delta{\pmb{S}}\| = 1. \nonumber
\end{align}

The optimization problem \eqref{eq:certgain} is difficult to solve in general; however, solvability condition \eqref{eq:ConditionInf} can estimate a lower bound called the certified admissible gain, $\lambda_{CAG}$. In our approach, we will characterize $\lambda_{CAG}$ with the help of Theorem \ref{theo:certlim}.

\begin{table}[ht]
\centering
  \begin{tabular}{| c || c | c |}
    \hline
    System size, N  & $\lambda_{CAG}/\lambda_R$ & $\lambda_{CAG}/\lambda_B$ \\ \hline
    3  & 0.7456 & 0.9507\\ \hline
    18  & 0.4102 & 0.6173\\ \hline
    33  & 0.5084 & 0.6629\\ \hline
    69  & 0.3123 & 0.4080\\ \hline
    123  & 0.5759 & 0.7708\\ \hline
  \end{tabular}
\caption{Certified gain limits vs. the estimated and true ones}
  \label{table:CAG}
\end{table}

\begin{theorem} \label{theo:certlim}
Let $\lambda_M$ be the larger positive root of the following equation
\begin{align}
& 2\sqrt{ \left(\norm{\pmb{M}_\s\conj{\pmb{Z}}_\s}+\norm{\pmb{N}_\s \pmb{Z}_\s}\right) \norm{\inv{\pmb{J}_\s}}\norm{\pmb{Z}_\s} \lambda \left(\norm{\pmb{S}_\s} + \lambda \right)} \nonumber \\
& + 2 \left(\norm{\pmb{M}_\s\conj{\pmb{Z}}_\s}+\norm{\pmb{N}_\s \pmb{Z}_\s}\right) \lambda = 1. \label{eq:strongcert}
\end{align}
Then, the certified admissible gain can be computed as 
\begin{equation}
\lambda_{CAG} = \min\{\lambda_M, 0.5/ \left(\norm{\pmb{M}_\s\conj{\pmb{Z}}_\s}+\norm{\pmb{N}_\s \pmb{Z}_\s}\right)\}.
\end{equation}
\end{theorem}

The proof of Theorem \ref{theo:certlim} is presented in Appendix \ref{proof:certlim}. Moreover, it is worth mentioning the certified gain is a lower bound of the optimal objective value of the min-max problem \eqref{eq:certgain}. The enforced solvability constraint ensures that the loading level $\pmb{S} = \pmb{S}_\s + \lambda_{CAG} \Delta{\pmb{S}}$ is solvable for all normalized loading directions $\Delta{\pmb{S}}$. In other words, the certified gain is independent of incremental loading directions but depends on the base operating point.

To validate the certified loading gain, we compare it with the estimated gain $\lambda_B$ from condition \eqref{eq:ConditionInf} and the actual gain limit, $\lambda_R$. We choose the loading direction where all loads increase equally, as the system will otherwise soon become stressed and the stability margin will decrease significantly. Table \ref{table:CAG} shows that, for this specific loading direction, the certified gain is around $30\%-80\%$ of the true gain limits. For the estimated gain from \eqref{eq:ConditionInf}, the ratios are higher, or even equal to $1$ in some cases. When the ratio reaches $1$, the condition \eqref{eq:ConditionInf} and its strong form \eqref{eq:strongcert} are equivalent along the homogeneous loading direction.

\section{Numerical studies}\label{sec:Numeric}

\subsection{Run-time analysis} \label{sec:efficiency}
To generate a solvability certificate \eqref{eq:ConditionInf}, the main burden is to explicitly compute for the matrix $\pmb{Z}_\s$ and $\pmb{J}_\s^{-1}$ that requires the inverse of the admittance matrix $\pmb{Y}$ and $\pmb{J}_\s$, respectively. In general, the computation effort involved in such inversion depends strongly on the properties of the network, and is hard to characterize a-priori, with empirical studies suggesting that the scaling is faster than  $\mathcal{O}(n^2)$ \cite{alvarado1976computational}. Instead of discussing the theoretical complexity, here we only assess the computational efficiency of generating certificates under the assumption that the impedance matrix and the inverse of the Jacobian are given. The matrix $\pmb{Z}_\s$ needs to be computed only once unless the network topology changes, and for a base operating point, the nominal matrix $\pmb{J}_\star$ and its inverse is fixed. Consequently, such matrices can be reused while repeatedly generating the certificate \eqref{eq:ConditionInf}, for example, in the screening problem considered in section \ref{sec:fastscr}. Then, a run-time analysis is performed to estimate the running time needed to generate a single certificate as the system size, $N$, increases. In particular, for each test case, we repeatedly generate the same certificate and measure the corresponding elapse time. The codes are written in MATLAB and implemented in a regular laptop with a configuration of 2.8 GHz Intel Core i7 CPU and 16 GB Memory. The average elapsed time to generate a certificate for $4-$node, $22-$node, $69-$node, and $141-$node test feeders are $1.25\cdot10^{-4}\,s$, $2.10\cdot10^{-4}\,s$, $9.10\cdot10^{-4}\,s$, and $3.44\cdot10^{-3}\,s$, respectively.

\subsection{A toy example and the coalescence condition} \label{sec:2bustheo}
Consider a $2$-bus test case with a slack bus with $V_0 =1\angle 0$, and one load bus with unknown voltage $V\angle\theta$ consuming an amount of apparent power $S = P + jQ$. The line connecting the two buses has an impedance of $R+jX$. In the base case, we have $S_\s = 0$ and $V_\star = V_0$. Applying the condition \eqref{eq:ConditionInf} to the $2$-bus system yields the inequality below:
\begin{align} \label{eq:est2bus}
&\sqrt{(R^2+X^2)(P^2+Q^2)}\leq \frac{1}{4}.
\end{align}

In the following simulations, we construct the estimated solvability boundaries and the real boundaries while varying the $R/X$ ratio. Very high $R/X$ ratios are not practical, yet we examine such extreme cases to illustrate the conditions for the coalescence between the approximated and the actual solvability boundaries.
\begin{figure}[ht]
    \centering
    \subfigure[$R/X = 0.1$]{{\includegraphics[scale = 0.26]{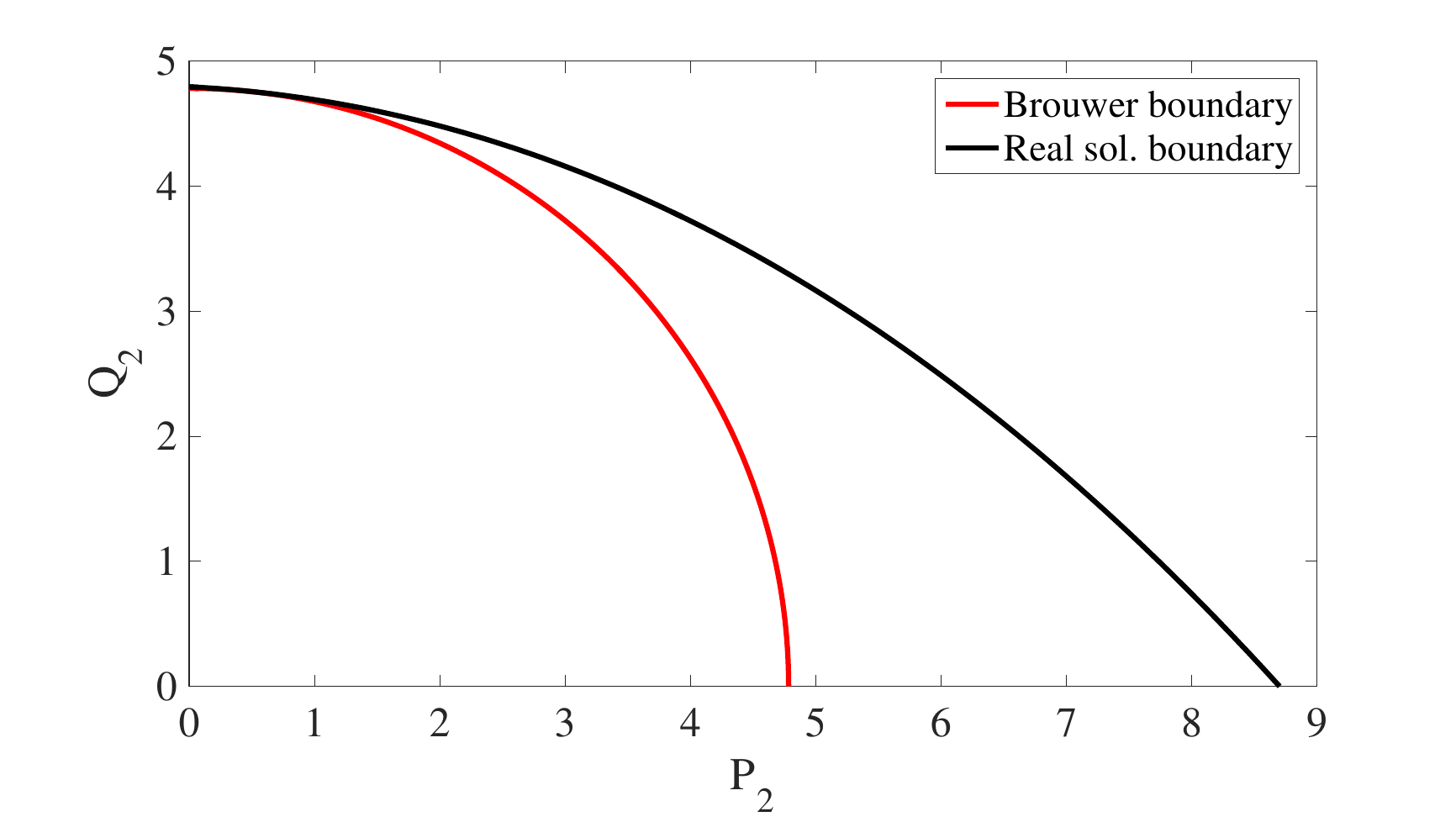}}}
   \subfigure[$R/X = 10$]{{\includegraphics[scale=0.26]{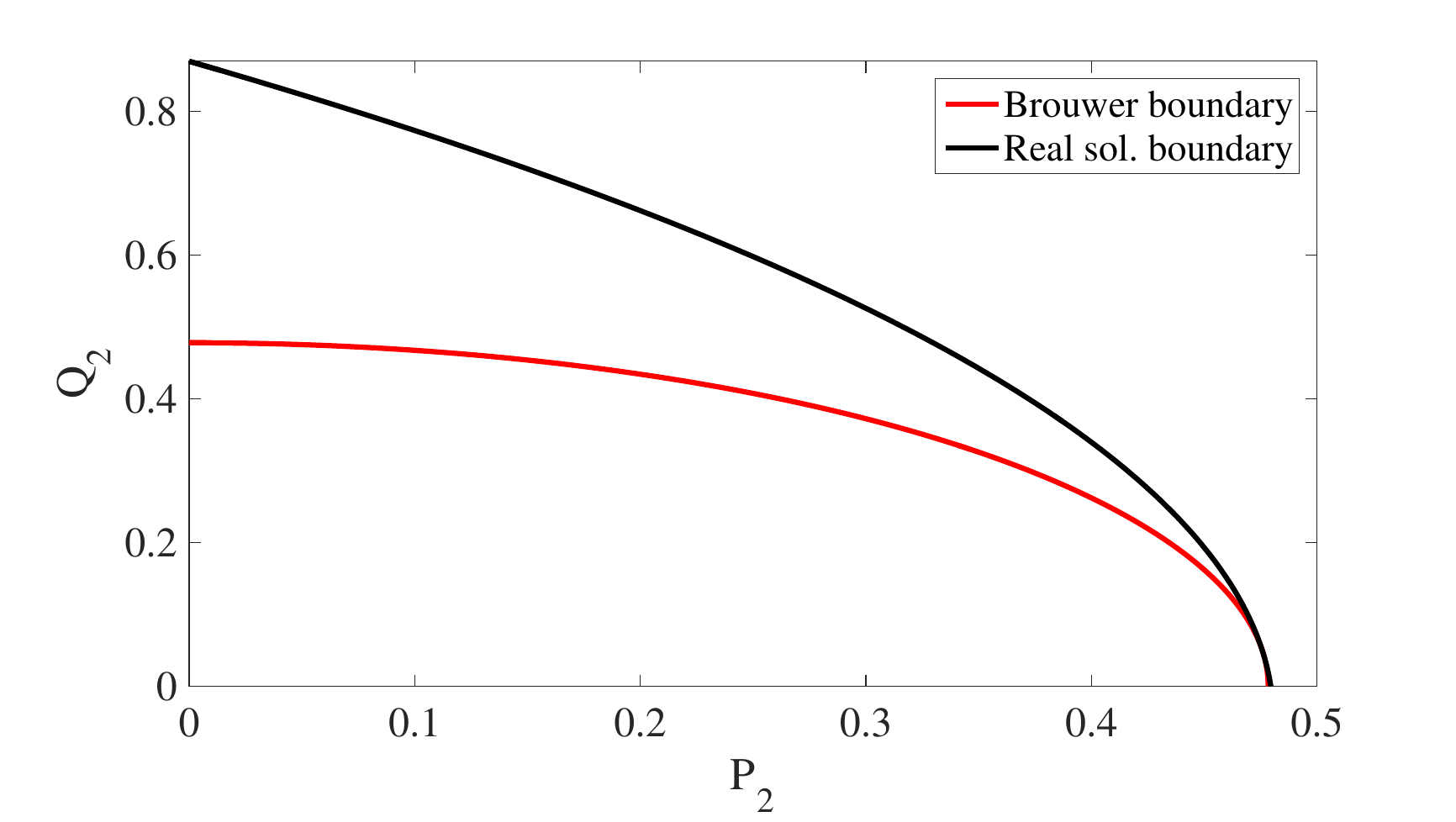}}}%
    \caption{Solvability regions of a two-node test feeder for a zero base load and different $R/X$ ratios}%
    \label{fig:2bus}%
\end{figure}
Figure \ref{fig:2bus} is plotted in $PQ$ space where the solid black curves represent the real boundaries, and the solid red curves represent the estimated boundaries using Brouwer approach. From Figure \ref{fig:2bus}, the most important observation is that the two boundaries may be tight in some directions that satisfies the condition $P/Q = R/X$, a ratio which we will refer to as the ``matching'' ratio. Note that the equation $P/Q = \cot(\phi)$ holds, where $\cos(\phi)$ is the load power factor. The coalescence condition is then proved as below.

For the $2$-bus toy system, the condition for the existence of a real solution can be expressed as \cite{Cutsem, vournas2015maximum}
\begin{align} \label{eq:real2bus}
(RQ-XP)^2 + RP+XQ \leq \frac{1}{4}.
\end{align}

Under the coalescence condition $\frac{R}{X} = \frac{P}{Q}$, both sufficient solvability condition \eqref{eq:est2bus} and the real condition \eqref{eq:real2bus} define the same solvability boundary characterized by $|RP + XQ| = \frac{1}{4}$.
\vspace*{-\baselineskip} 
\begin{figure}[ht]
    \centering
     \includegraphics[width=1.\columnwidth]{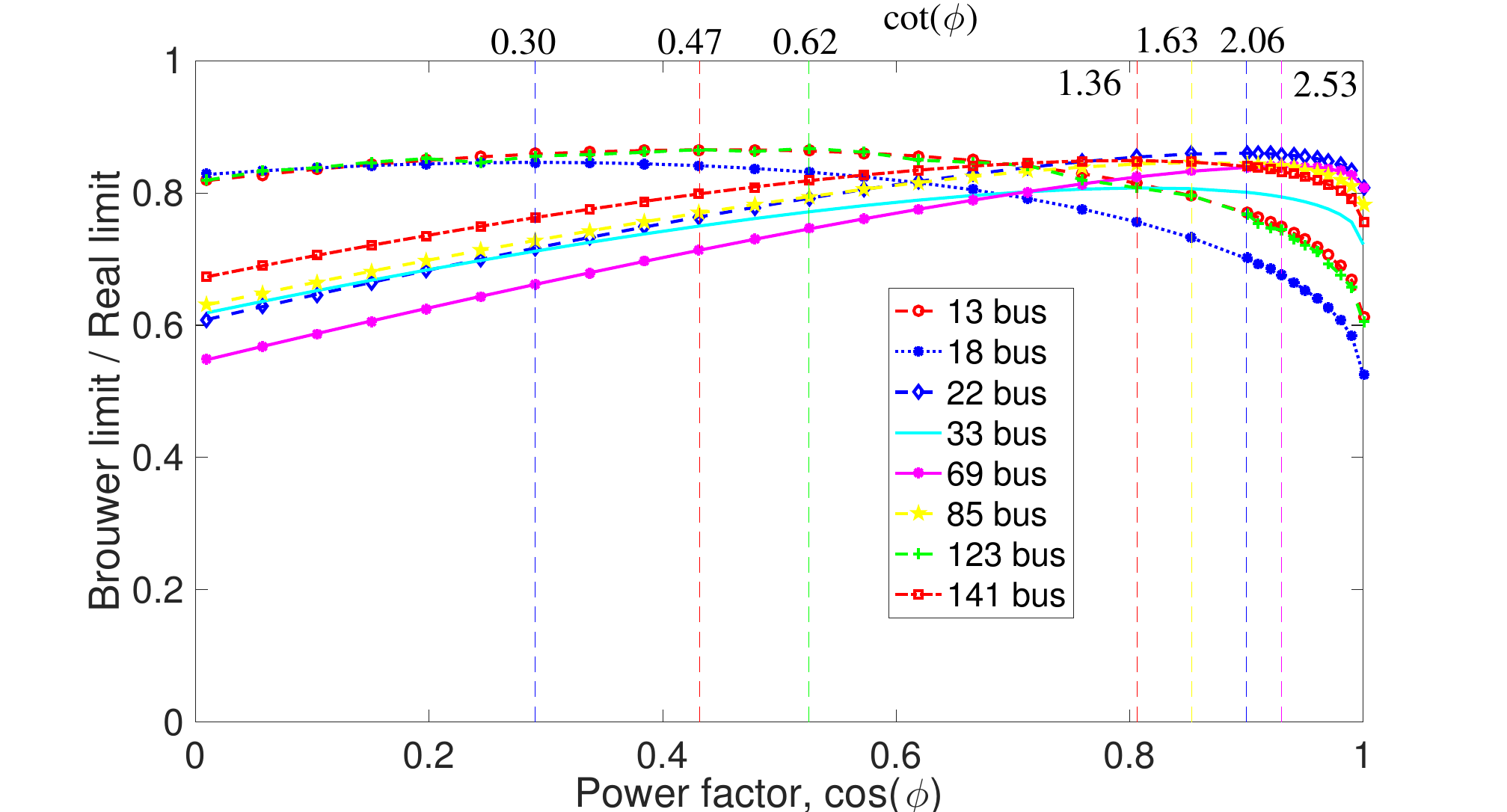}
	\caption{Solvability regions comparisons with different power factors}
     \label{fig:coalescence}
\end{figure}

\begin{table}[ht]
\centering 
\begin{tabular}{| c| c| c|c| }
\hline
 Test feeder & Matching $P/Q$ & Average  & Common \\
 & $cot(\phi)$ & $R/X$ &  $R/X$ range\\
 \hline
 $33$-node & 1.36 & 1.44 & 1.07-1.36\\
 $69$-node & 2.53 & 2.07 & 2.80-3.10\\
 $123$-node & 0.62 & 0.63 & 0.40-0.62\\
 $141$-node & 1.36 & 1.71 & 1.16-1.74\\
 \hline
\end{tabular}
\caption{Coalescence condition: matching $P/Q-R/X$ ratios}
\label{table:matching}
\end{table}

For large-scale distribution systems, unfortunately, we can neither observe any case where the estimated boundary matches the actual one nor provide any rigorously mathematical proof or estimation regarding the gap between the two boundaries. However, extensive simulation results reveal that the coalescence condition ``almost'' holds for large systems with a zero-loading base point. Figure \ref{fig:coalescence} plots the covering ratio--which we define as the ratio of the estimated loadability limit to the real loadability limit--against the homogeneous power factors $\cos(\phi)$ and $\cot(\phi)$. The intersection point of each pair, consisting of the horizontal curve and the vertical line of the same colour, indicates the maximum ratio for the corresponding distribution system. All maximum ratios are larger than $0.8$, indicating that the estimated solvability limit can cover more than $80\%$ of the actual limit. Moreover, the matching $P/Q$ ratio or $\cot(\phi)$ can be approximated by the average $R/X$ ratio of the lines. In all considered test cases, with $P/Q = \langle R/X\rangle$, the maximum covering percentage is circa $80\%$. It can also be seen that, if the lines are almost homogeneous in terms of the $R/X$ ratio, the matching ratio likely falls within the most common range of $R/X$ as shown in Table \ref{table:matching} for the $33$-node, $69$-node, $123$-node, and $141$-node test feeders.

\subsection{Large-scale distribution feeders} \label{sec:largesys}

\begin{figure}[ht]
    \centering  \includegraphics[width=1\columnwidth]{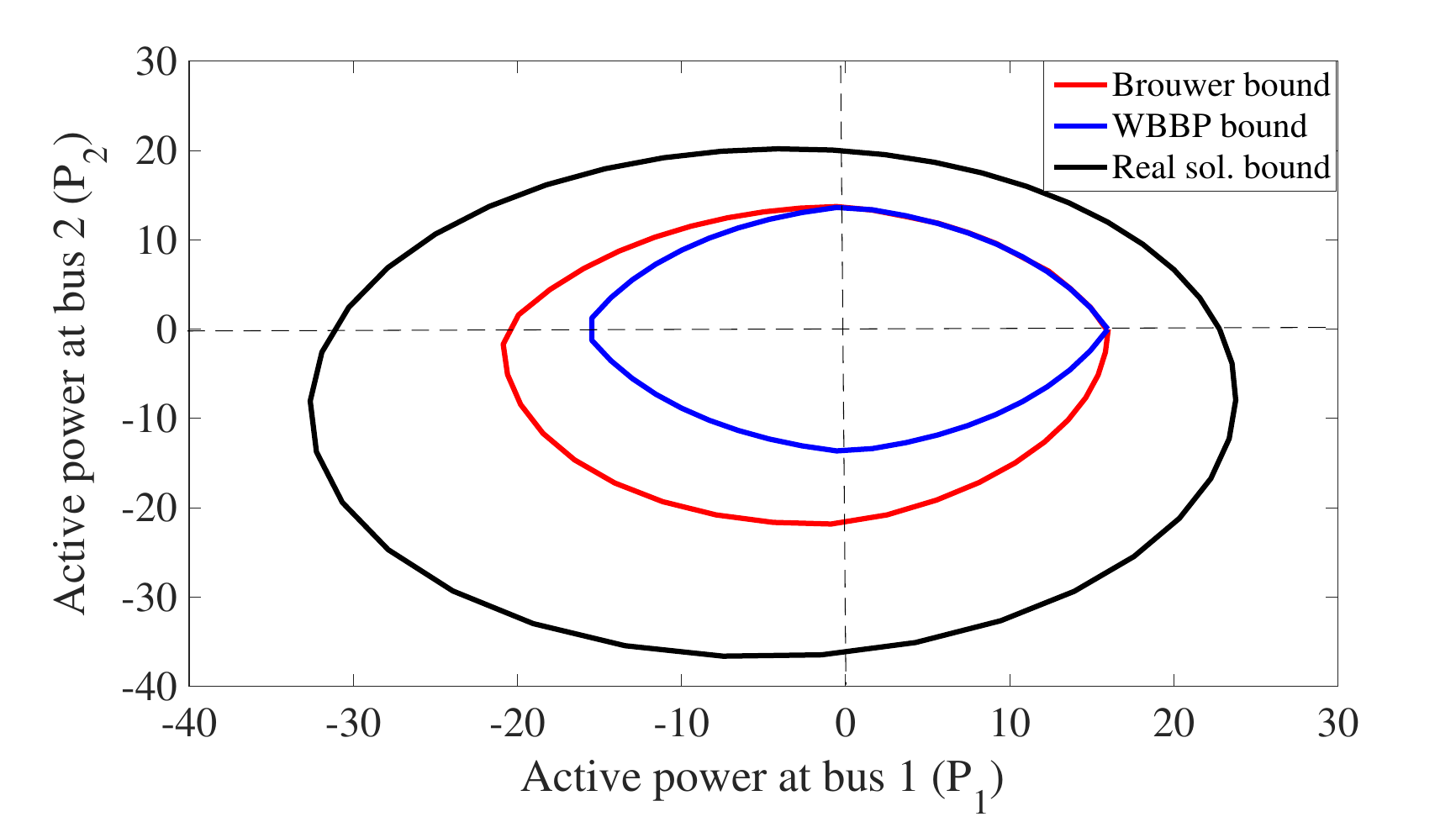}
	\caption{Solvability regions of IEEE $123$-node test feeder for a zero-loading base case}
     \label{fig:sol123full}
\end{figure}


\begin{figure}[ht]
    \centering  \includegraphics[width=1\columnwidth]{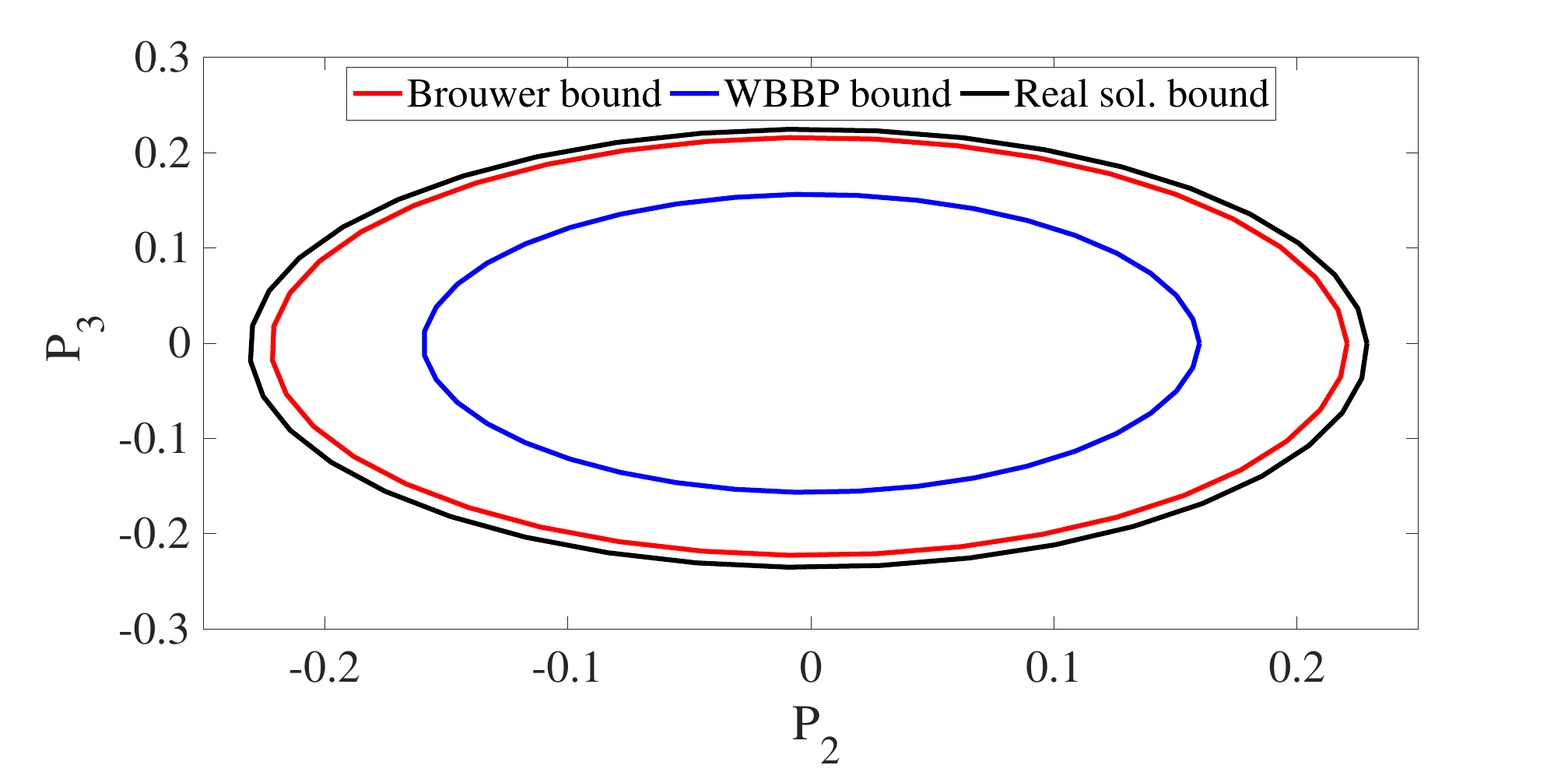}
	\caption{Solvability regions of IEEE $141$-node test feeder, $P/Q \approx \langle R/X \rangle$}
     \label{fig:sol141bus}
\end{figure}

\begin{figure}[ht]
    \centering  \includegraphics[width=1.\columnwidth]{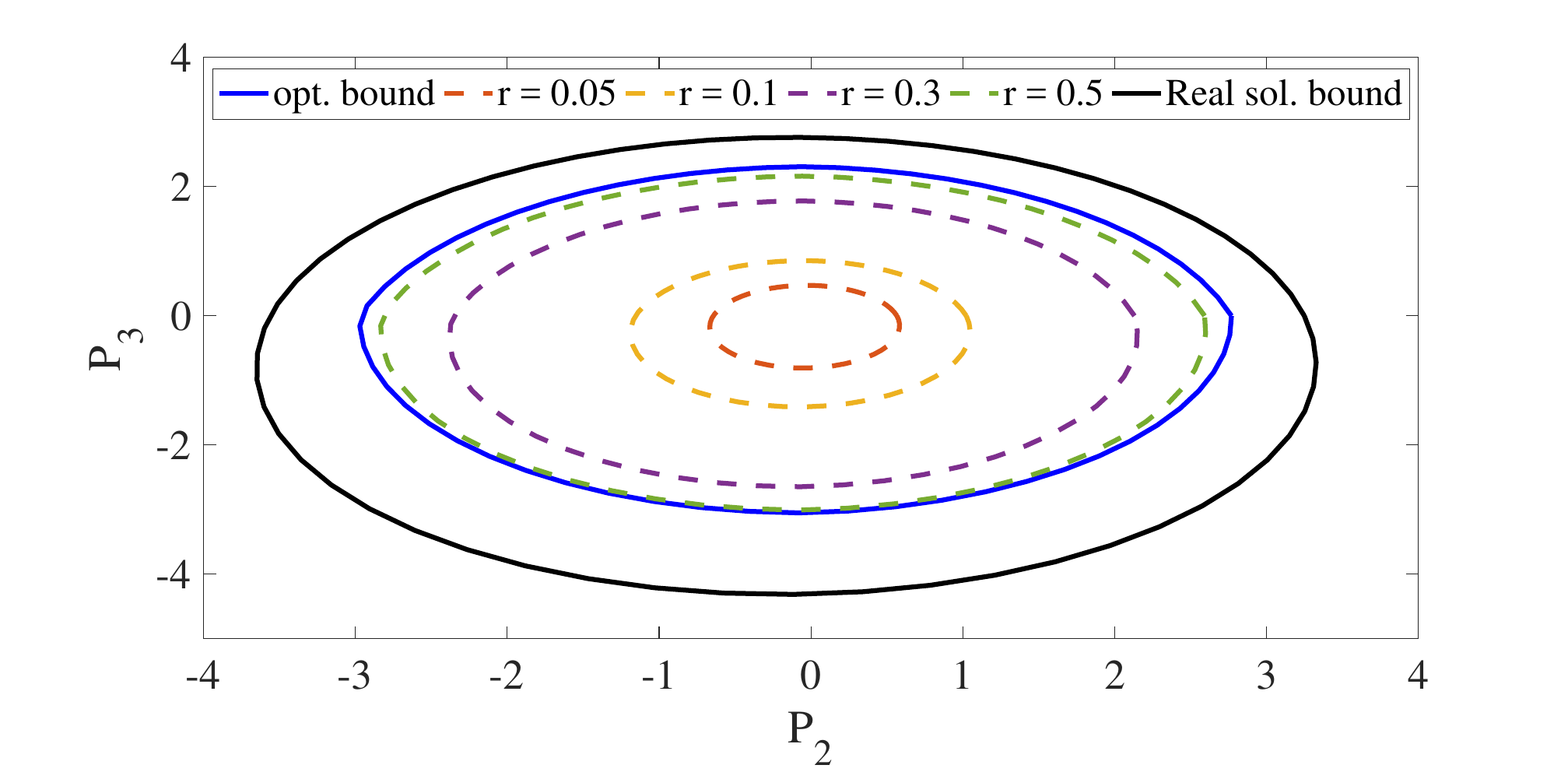}
	\caption{Solvability regions of IEEE $141$-node test feeder with different radii $r$ of voltage bounds}
     \label{fig:multir141}
\end{figure}

In this section, we continue constructing the solvability regions for larger scale test feeders. Some of those test cases are provided in MATPOWER package \cite{IEEEtestcaseMATPOWER}. Moreover, we use only per-unit system with the detailed information is included in the corresponding MATPOWER test cases. For example, for the $22$-bus case, the base quantities are $10$ MVA and $11$ kV. Note that most of our simulations consider incremental loading scenarios which maintain a constant power factor of $\cos(\phi)=0.9$, or $\Delta P_i/\Delta S_i = 0.9$ with bus $i$ is a load bus. The effect of the power factor is discussed in section \ref{sec:2bustheo}. Apart from that, the real loadability limits for all considered test feeders except the $2$-bus toy system are computed using the traditional continuation power flow technique. For a normal loading condition, Figures \ref{fig:sol123full}, \ref{fig:sol141bus} show that the approximated solvability region is large enough compared to the actual loadability region. Figure \ref{fig:sol141bus} illustrates the coalescence condition where the relation $P/Q \approx \langle R/X \rangle$ holds, implying that the estimated boundary almost matches the actual one. For this simulation, the base loading level is doubled the demand provided in the MATPOWER package. Moreover, we choose the incremental loading vector such that $\Delta P_i/\Delta S_i = 0.867$, $\Delta P_i = 0.1 \Delta P_1$ with $i \not\in \{2,3\}$ where bus $i$ is a load bus, and $\Delta P_2 = \Delta P_3$.

In addition, if the voltage bound of the solutions is of interest, one can construct the solvability region using \eqref{eq:Condition} with the corresponding radius $r$, where $0 \leq r < 1$. The volume of the union of the regions $\mathcal{U}_r$ increases with $r$, implying that, if one imposes a tighter bound on the voltage solution, then the power injections need also be confined inside a smaller set. Figure \ref{fig:multir141} clearly confirms the preceding statement.

The solvability criterion is constructed using norm-constrained bounds which, define a region in parameter space where the system has power flow solutions. As visualizing the corresponding solvability regions in multi-dimensional space is challenging, we only plot their cross-section in the plane of two different parameters. In each 2D plane, we consider a series of directions characterized by different incremental loading vector $\Delta \pmb{S}$, along which we trace the maximum power level following the estimation technique described in Appendix \ref{app:estimation}. Then the estimated solvability boundaries in blue are easily created by connecting all estimated maximum points. Meanwhile, the actual loadability and WBBP's boundaries are also constructed using the same set of loading directions.

We also compare our method to that of WBBP introduced in \cite{wang2016existence}. In most of the cases, the solvable regions constructed using Brouwer's approach encompass WBBP's. Figure \ref{fig:sol123full} shows that, for a zero-loading base point, our certificate is identical to that of WBBP's in the consumption regime, but it starts dominating as the loads inject powers. In practice, the injecting condition can be realized with distributed generation. There are some special cases such as in Figure \ref{fig:sol123bus}(a) wherein Brouwer boundary is much larger than that of WBBP's. In this particular simulation, the load consumes more reactive power than active power at the base operating point. From a mathematical perspective, such base point voltage is close to the limit where WBBP's solvability condition is no longer valid. As a result, the characterized region becomes extremely conservative. In contrast, our estimation is still functional. However, there exist some regimes where WBBP's method outperforms ours. An example of such cases is shown in Figure \ref{fig:sol123bus}(b) where the system is much more stressed than usual. In these simulations, the base case has a homogeneous power factor (i.e., all loads have the same power factor). This power factor is $0.2425$ and $0.2316$ for the cases illustrated in \ref{fig:sol123bus}(a) and \ref{fig:sol123bus}(b), respectively. Moreover, for the incremental loading vector, we fix the power factor $\Delta P_i/\Delta S_i = 0.9$, and all loads at bus $i = 3,\dots$ will increase by a factor of $0.1$, or in short, $\Delta P_i = 0.1 \Delta P_1$ for $i \not\in \{1,2\}$. We further imposed the condition $\Delta P_1 = \Delta P_2$.

\begin{figure}[ht]
    \centering
    \subfigure[Brouwer's results are superior]{{\includegraphics[scale = 0.26]{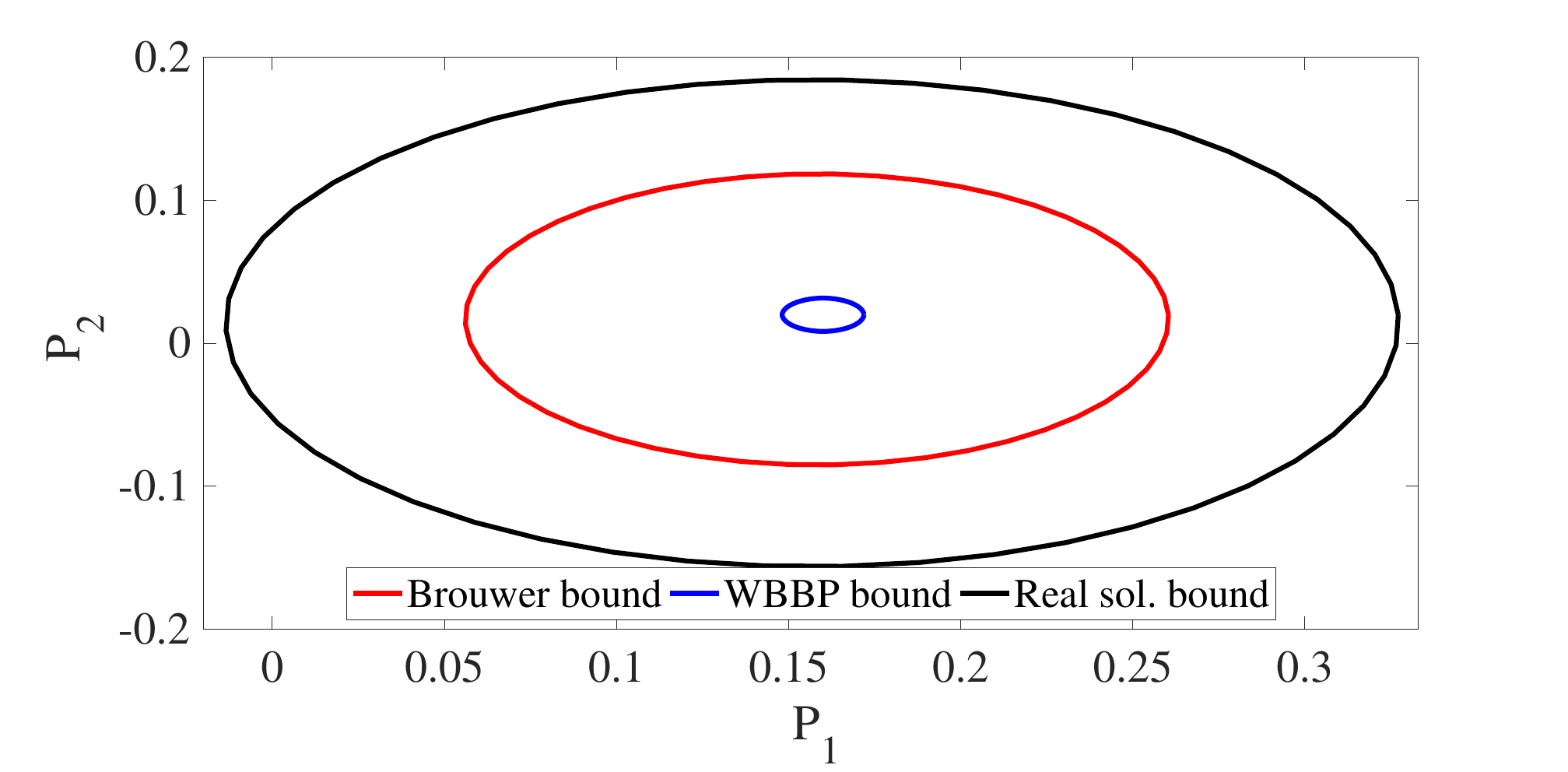}}}
    \subfigure[WBBP's results are superior]{{\includegraphics[scale = 0.26]{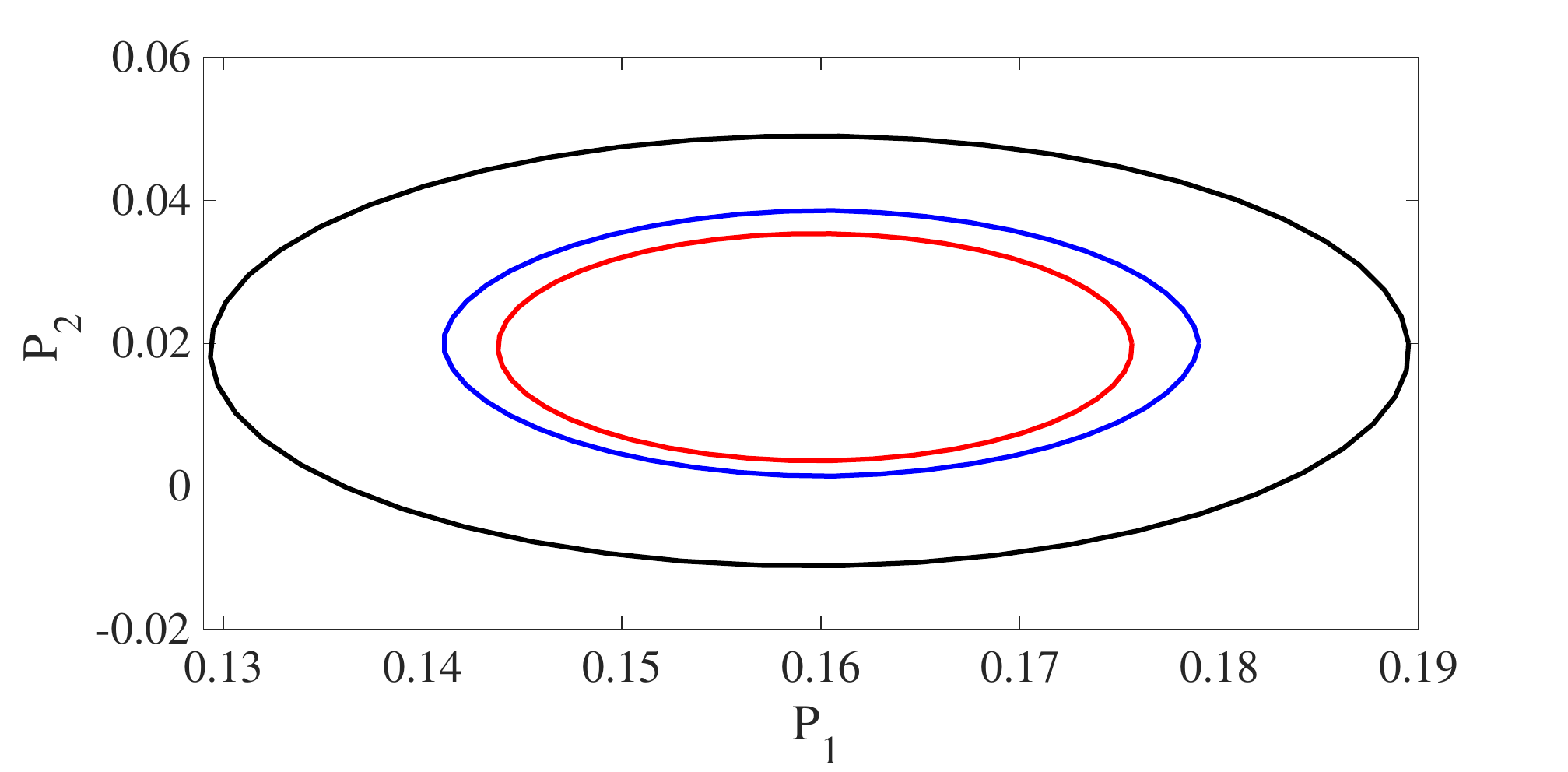}}}%
    \caption{Solvability regions of IEEE $123$-node test feeder \cite{IEEEtestcase123} for nonzero loading base cases}%
    \label{fig:sol123bus}%
\end{figure}

\begin{figure}[ht]
    \centering
    \subfigure[WBBP's/Brouwer's limit ratio]{{\includegraphics[scale=0.26]{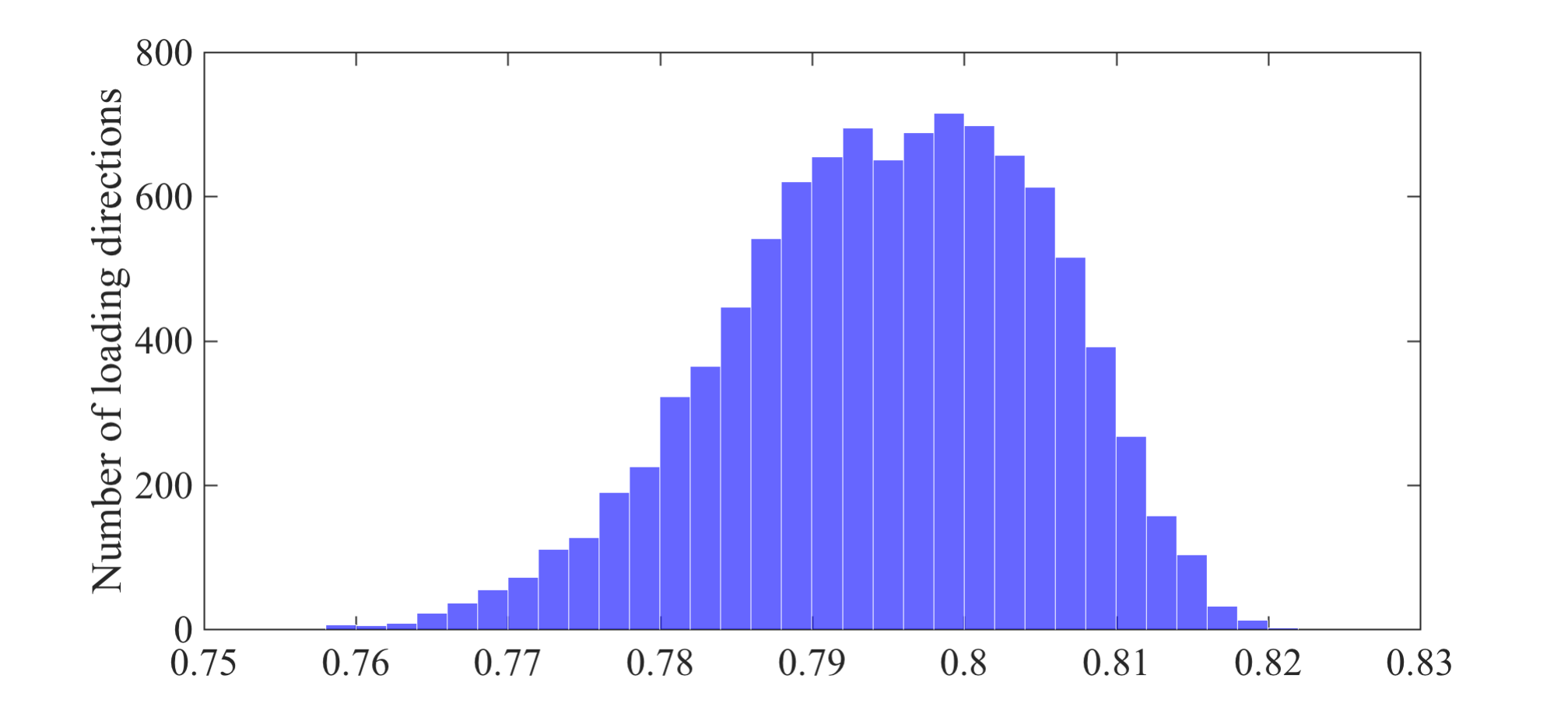}}}
   \subfigure[Brouwer's/Real limit ratio]{{\includegraphics[scale=0.26]{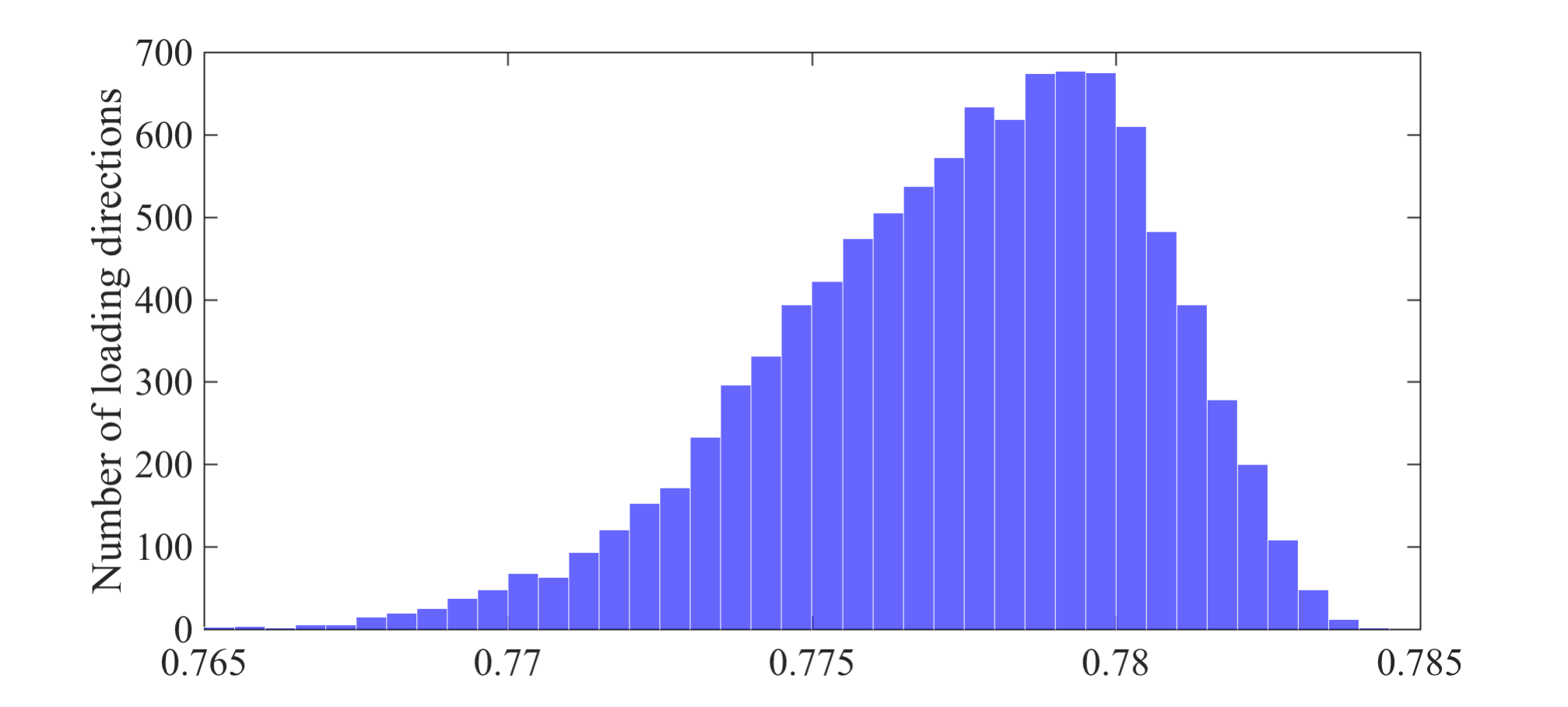}}}%
    \caption{Comparison of WBBP's, Brouwer's, and real loading limits for IEEE $141$-node test feeder}%
    \label{fig:hist}%
\end{figure}

The relative performance comparison between the two methods is extensively analyzed for the modified $141$-node test feeder with renewable penetration used in section \ref{sec:probmeas}. We assessed $10,000$ random loading scenarios in which the load buses with PVs can either inject or consume powers. The histograms plotted in Figure \ref{fig:hist} show that WBBP's results can cover circa $80\%$ of ours, which in turn cover almost the same percentage of the real limits.

\section{Conclusions and Future work}
In this paper, we developed an inner approximation technique for constructing convex subsets of the solvability region for distribution systems based on Brouwer's fixed-point theorem. The constructed regions can be used in security-related functions that rely on steady-state snapshots. In particular, the proposed fast screening tool based on the sufficient solvability conditions was shown to have a considerably faster screening process compared to that of the conventional screening process. Meanwhile, we introduced a new stability indicator, the certified admissible gain limit, which represents a safe gain wherein the system may move along any incremental loading directions without exhibiting voltage instability.

As mentioned earlier, this work focuses only on distribution systems with $PQ$ loads, and only simple bounds on the voltage solution can be imposed. For future research, we plan to extend the proposed technique to handle voltage-controlled buses in transmission networks and incorporate operational constraints such as voltage and current limits. In addition, it is possible to include $3$-phase AC distribution systems in the proposed framework.

\section{Acknowledgement}
Work of KT was funded by DOE/GMLC 2.0 project: ``Emergency monitoring and controls through new technologies and analytics''. KT, SY, and HN were supported by the NSF awards 1554171 and 1550015, Siebel Fellowship, and VEF. The work of HN was supported by NTU SUG. KD was supported by project 1.10 in the Control of Complex Systems Initiative, a Laboratory Directed Research and Development (LDRD) program at the Pacific Northwest National Laboratory. Also, we thank Paul Hines and Daniel Molzahn for their constructive feedback on our original work.

\bibliographystyle{IEEEtran}


\section{Appendix}
\subsection{Lemma \ref{lem:PFrewrite}}
\begin{lemma}\label{lem:PFrewrite}
\eqref{eq:PFbasic} can be rewritten as
\begin{align}
& \pmb{y}+\zeta\br{\pmb{S}_\s} \conj{\pmb{y}} \nonumber\\
&\qquad=-\eta\br{\Delta \pmb{S}}-\diag{\eta\br{\Delta \pmb{S}}}\pmb{y}-\zeta\br{\Delta \pmb{S}} \conj{\pmb{y}}-\diag{\pmb{y}}\zeta\br{\pmb{S}} \conj{\pmb{y}}
\end{align}
where $\pmb{y}=[\![\pmb{V}]\!]^{-1}\pmb{V}_\s -\One, \pmb{\gamma}_\s= [\![\pmb{V}_\s]\!]^{-1}\pmb{V}^0 $.
\end{lemma}
\begin{proof}[Proof (Lemma \ref{lem:PFrewrite})]

\eqref{eq:PFbasic} can be rewritten as 
\begin{equation}
    \inv{\conj{\pmb{Y}}}\inv{\diag{\pmb{V}}}\pmb{S}=\conj{\pmb{V}-\pmb{V}^0}.
\end{equation}
Multiplying by $\inv{\diag{\conj{\pmb{V}}}}$ on the left, we get
\begin{equation}
    \inv{\diag{\conj{\pmb{V}}}}\inv{\conj{\pmb{Y}}}\inv{\diag{\pmb{V}}}\pmb{S}=\One-\inv{\diag{\conj{\pmb{V}}}}\conj{\pmb{V}^0}.
\end{equation}

{After leaving only $\One$ on the RHS, factor out $[\![ \overline{\pmb{V}} ]\!]^{-1}$ so that we are left with the equation below: }
{ 
\begin{align}
&[\![\overline{\pmb{V}}_\s]\!][\![\overline{\pmb{V}}]\!]^{-1}
\left( [\![\overline{\pmb{V}}_\s]\!]^{-1} \overline{\pmb{Y}}^{-1} [\![\pmb{V}_\s]\!]^{-1} [\![\pmb{S}]\!] [\![\pmb{V} ]\!]^{-1} \pmb{V}_\s 
+ [\![\overline{\pmb{V}}_\s]\!]^{-1} \overline{\pmb{V}}^0 \right)\nonumber\\
&= \One.
\end{align}
}

{Substituting in the values for $\pmb{Z}_\s$ and $\pmb{\gamma}_\s$, and defining $x = [\![\conj{{\pmb{V}}}]\!]^{-1} \conj{\pmb{V}}_\s$}, the above equation reduces to
\begin{equation}
    \diag{\pmb{x}}(\pmb{Z}_\s \diag{\pmb{S}}\conj{\pmb{x}} + \conj{\pmb{\Vinv_\s}}) = \One.
\end{equation}
Let $\pmb{y}=\conj{\pmb{x}}-\One$. Conjugating the above equation, we obtain
\begin{equation}
    \diag{\One+\pmb{y}}(\zeta\br{\pmb{S}} \br{\One+\conj{\pmb{y}}} + \pmb{\Vinv}_\s) = \One
\end{equation}
which {expands to}
\begin{equation}
 \diag{\pmb{y}}\zeta\br{\pmb{S}} \One+\zeta\br{\pmb{S}} \conj{\pmb{y}}+\diag{\pmb{y}}\zeta\br{\pmb{S}} \conj{\pmb{y}}  + \diag{\pmb{\Vinv}_\star}\pmb{y}+\eta\br{\Delta \pmb{S}} = 0
\end{equation}
where we use the relation that $\eta\br{\Delta \pmb{S}} = \zeta\br{\pmb{S}}\One+\pmb{\Vinv}_\star - \One$. This can be rewritten as
\begin{equation}
 \diag{\eta\br{\pmb{S}}+\pmb{\gamma}_\s}\pmb{y}+\zeta\br{\pmb{S}} \conj{\pmb{y}}+\diag{\pmb{y}}\zeta\br{\pmb{S}} \conj{\pmb{y}}  + \diag{\pmb{\Vinv}_\star}\pmb{y}+\eta\br{\Delta \pmb{S}} = 0. \label{eq:tmpEq}
\end{equation}
{Meanwhile, we know that, by the definition of $\eta(\pmb{S})$ and $\pmb{\gamma}$,}
\begin{align}
\eta\br{\pmb{S}}+\pmb{\gamma}_\s
&=\conj{\pmb{Z}_\s \pmb{S}_\s}+\inv{\diag{\pmb{V}_\s}}\pmb{V}^0 \\
& = \inv{\diag{\pmb{V}_\s}} \inv{\pmb{Y}} \br{\inv{\diag{\conj{\pmb{V}_\s}}} \conj{\pmb{S}_\s}+ \pmb{Y} \pmb{V}^0} \\
& = \inv{\diag{\pmb{V}_\s}} \inv{\pmb{Y}}\br{ \pmb{Y} \pmb{V}_\s} = \One.
\end{align}
Thus, \eqref{eq:tmpEq} can be rewritten as
\begin{align}
 &\pmb{y} + \zeta\br{\pmb{S}_\s} \conj{\pmb{y}} \nonumber\\
 &\qquad=  - \eta\br{\Delta \pmb{S}}- \diag{\eta\br{\Delta \pmb{S}}} \pmb{y} - \zeta\br{\Delta \pmb{S}} \conj{\pmb{y}}  - \diag{\pmb{y}}\zeta\br{\pmb{S}} \conj{\pmb{y}}.
\end{align}
Q.E.D.

\subsection{Proof of Theorem \ref{theo:certlim}} \label{proof:certlim}
By applying basic properties of operator norms, i.e., $\norm{A + B} \leq \norm{A} + \norm{B}$ and $\norm{A B} \leq \norm{A} \norm{B}$, we can derive an upper bound of the left hand side of condition \eqref{eq:ConditionInf} as below

\begin{align}
& 2\sqrt{ \left(\norm{\pmb{M}_\s\conj{\pmb{Z}}_\s}+\norm{\pmb{N}_\s \pmb{Z}_\s}\right) \norm{\inv{\pmb{J}_\s}}\norm{\pmb{Z}_\s} \norm{\Delta \pmb{S}} \norm{\pmb{S}}} \nonumber \\
& + 2 \left(\norm{\pmb{M}_\s\conj{\pmb{Z}}_\s}+\norm{\pmb{N}_\s \pmb{Z}_\s}\right) \norm{\Delta \pmb{S}} \nonumber
\end{align}
with the help of $\norm{\diag{\conj{\Delta \pmb{S}}}} =\norm{\diag{\Delta \pmb{S}}} = \norm{\Delta \pmb{S}}$. Furthermore, we have that $\norm{\pmb{S}} = \norm{\pmb{S}_\s + \Delta \pmb{S}} \leq \norm{\pmb{S}_\s} + \norm{\Delta \pmb{S}}$, and $\norm{\Delta \pmb{S}} = \norm{\lambda \Delta \pmb{u}} \leq \lambda$ as $\norm{\Delta \pmb{u}}= 1$. Then we arrive at a stronger form of \eqref{eq:ConditionInf}:
\begin{align}
& 2\sqrt{ \left(\norm{\pmb{M}_\s\conj{\pmb{Z}}_\s}+\norm{\pmb{N}_\s \pmb{Z}_\s}\right) \norm{\inv{\pmb{J}_\s}}\norm{\pmb{Z}_\s} \lambda \left(\norm{\pmb{S}_\s} + \lambda \right)} \nonumber \\
& + 2 \left(\norm{\pmb{M}_\s\conj{\pmb{Z}}_\s}+\norm{\pmb{N}_\s \pmb{Z}_\s}\right) \lambda \leq 1. \label{eq:strongcertin}
\end{align}
Any variation $\Delta{\pmb{S}} = \lambda \Delta{\mathbf{u}}$ that satisfies \eqref{eq:strongcertin} will also satisfy \eqref{eq:ConditionInf}. Letting the inequality \eqref{eq:strongcertin} hold as an equality, yields \eqref{eq:strongcert}. Moreover, to guarantee the real non-negativity of the square root term in \eqref{eq:strongcert}, it requires that 
\begin{equation}
2 \left(\norm{\pmb{M}_\s\conj{\pmb{Z}}_\s}+\norm{\pmb{N}_\s \pmb{Z}_\s}\right) \lambda \leq 1,
\end{equation}
or
\begin{equation}
\lambda \leq 0.5/ \left(\norm{\pmb{M}_\s\conj{\pmb{Z}}_\s}+\norm{\pmb{N}_\s \pmb{Z}_\s}\right).
\end{equation}

Then if $\lambda_M$ is the larger positive root of the equality \eqref{eq:ConditionInf}, we can calculate below the certified gain with which the system can be loaded along all normalized incremental direction without encountering voltage collapse:
\begin{equation}
\lambda_{CAG} = \min\{\lambda_M, 0.5/ \left(\norm{\pmb{M}_\s\conj{\pmb{Z}}_\s}+\norm{\pmb{N}_\s \pmb{Z}_\s}\right)\}.
\end{equation}

\subsection{Theoretical comparison to WBBP's results for the nominal solution} \label{app:nullinject}
Here we show that, for the nominal solution $\bf{V}^\s=\mathbf{1},\pmb{S}^\star=0$, our solvability certificate implies that by WBBP. In other words, for such nominal base point, our inner approximation encompasses WBBP's. To prove this, we introduce the related quantities used in \cite{wang2016existence}, namely $\bf{w} =\bf{V}^\s$, and $\zeta\br{\pmb{S}}=\inv{\diag{\bf{w}}} \inv{\bf{Y}}\inv{\diag{\bf{w}}}\diag{\pmb{S}}$. Note that for the zero power condition, we have $\bf{V}^\s =\conj{\bf{V}^\s} =\mathbf{1}$, and $\zeta\br{\pmb{S}}= \conj{\bf{Z}_\s} \diag{\pmb{S}}$. Moreover, the base Jacobian is an identity matrix, thus $\bf{M}_\s = \bf{1}$ and $\bf{N}_\s = 0$. The solvability criterion \eqref{eq:ConditionInf} duly becomes:
\begin{align}
2\norm{\zeta\br{\pmb{S}}\mathbf{1}}_\infty+2\sqrt{\norm{\zeta\br{\pmb{S}}\mathbf{1}}_\infty\norm{\zeta\br{\pmb{S}}}_\infty} \leq 1. \label{eq:EPFLCertFinal}
\end{align}

If we use the bound $\norm{\zeta\br{\pmb{S}}\mathbf{1}}_\infty\leq \norm{\zeta\br{\pmb{S}}}_\infty$, this condition is implied by the condition
\begin{align}\norm{\zeta\br{\pmb{S}}}_\infty\leq \frac{1}{4}\label{eq:EPFL}
\end{align}
 which is the region defined by Corollary 1 in \cite{wang2016existence}. Thus our analysis produces a stronger result than that from \cite{wang2016existence} for the zero power operating point.

\subsection{Maximum loading gain estimation} \label{app:estimation}
For a given base point $\pmb{S}_\s$ and a loading direction $\Delta \pmb{u}$, one needs to compute the maximum solvable loading level. A lower bound of such maximum level can be computed as $\pmb{S}_\s + \lambda_B \Delta \pmb{u}$, where $\lambda_B$ is the maximum gain satisfying the sufficient solvability condition \eqref{eq:ConditionInf}, i.e.

\begin{align} \label{eq:maxgainB}
& 2\sqrt{\norm{\pmb{M}_\s\conj{\pmb{Z}}_\s\diag{\conj{\Delta \pmb{u}}}+\pmb{N}_\s \pmb{Z}_\s \diag{\Delta \pmb{u}}}\norm{\inv{\pmb{J}_\s}}\norm{\pmb{Z}_\s \diag{\pmb{S}_\s + \lambda \Delta \pmb{u}}} \lambda} \nonumber \\
& +\norm{\pmb{M}_\s\conj{\pmb{Z}_\s} \diag{\conj{\Delta \pmb{u}}}+\pmb{N}_\s\diag{\pmb{Z}_\s\br{\Delta \pmb{u}}}} \lambda \nonumber \\
& +\norm{\pmb{M}_\s\diag{\conj{\pmb{Z}_\s\br{\Delta \pmb{u}}}}+\pmb{N}_\s \pmb{Z}_\s\diag{\Delta \pmb{u}}} \lambda \leq 1 
\end{align} 

In the condition above, all terms with star marks are given for a base point, and the loading direction $\Delta \pmb{u}$ is assumed to be known. Furthermore, with the help of the triangle inequality $\norm{\pmb{Z}_\s \diag{\pmb{S}_\s + \lambda \Delta \pmb{u}}} \leq \norm{\pmb{Z}_\s \diag{\pmb{S}_\s}} + \lambda \norm{\pmb{Z}_\s \diag{\Delta \pmb{u}}}$, one can obtain a stricter condition of \eqref{eq:maxgainB} which can be easily further transformed into a quadratic inequality in variable $\lambda$. Let such a quadratic inequality hold as equality, then one can find at most two corresponding solutions. $\lambda_B$ will be the largest solution which satisfies
\begin{align}
& \norm{\pmb{M}_\s\conj{\pmb{Z}_\s} \diag{\conj{\Delta \pmb{u}}}+\pmb{N}_\s\diag{\pmb{Z}_\s\br{\Delta \pmb{u}}}}  \lambda \nonumber\\ 
+ &\norm{\pmb{M}_\s\diag{\conj{\pmb{Z}_\s\br{\Delta \pmb{u}}}}+\pmb{N}_\s \pmb{Z}_\s\diag{\Delta \pmb{u}}}  \lambda \leq 1.
\end{align} 

The condition above is imposed simply to ensure the non-negativity of the square root term associated with the stricter condition of \eqref{eq:maxgainB}.
\end{proof}

\begin{IEEEbiography}[{\includegraphics[width=1in,height=1.25in,clip,keepaspectratio]{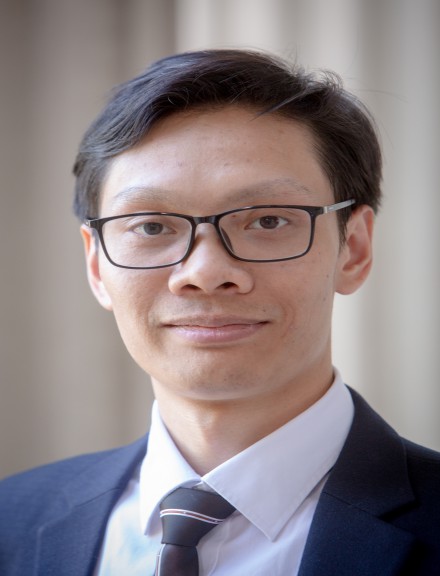}}]{Hung D. Nguyen} (S`12) received the B.E. degree in electrical engineering from HUT, Vietnam, in 2009, the M.S. degree in electrical engineering from Seoul National University, Korea, in 2013, and the Ph.D. degree in Electric Power Engineering at Massachusetts Institute of Technology (MIT). Currently, he is an Assistant Professor in Electrical and Electronic Engineering at NTU, Singapore. His current research interests include power system operation and control; the nonlinearity, dynamics and stability of large scale power systems; DSA/EMS and smart grids.
\end{IEEEbiography}
 \vspace{-10 mm}
\begin{IEEEbiography}[{\includegraphics[width=1in,height=1.25in,clip,keepaspectratio]{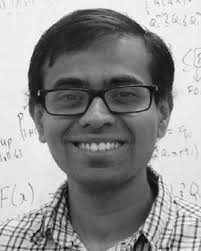}}]{Krishnamurthy Dvijotham} received the bachelor’s degree from the Indian Institute of Technology Bombay, Mumbai, India, in 2008, and the Ph.D. degree in computer science and engineering from the University of Washington, Seattle, Washington, USA, in 2014. He was a Researcher with the Optimization and Control Group, Pacific Northwest National Laboratory and a Research Assistant Professor with Washington State University, Pullman, WA, USA. He currently is a research scientist in Google Deepmind. His research interests are in developing efficient algorithms for automated decision making in complex physical, cyber-physical and social systems. 
\end{IEEEbiography}
\vspace{-10 mm}
\begin{IEEEbiography}[{\includegraphics[width=1in,height=1.25in,clip,keepaspectratio]{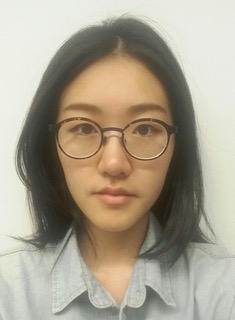}}]{Suhyoun Yu} received the BA in Mathematics, Wellesley College, and the Master in Mechanical Engineering, MIT. Currently, she is a Ph.D. candidate in Mechanical Engineering, MIT, with research in Reinforcement learning in the context of water and energy micro-nexus.
\end{IEEEbiography}

 \vspace{-10 mm}
\begin{IEEEbiography}[{\includegraphics[width=1in,height=1.25in,clip,keepaspectratio]{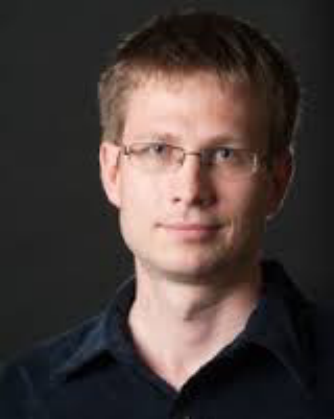}}]{Konstantin Turitsyn} (M`09) received the M.Sc. degree in physics from Moscow Institute of Physics and Technology and the Ph.D. degree in physics from Landau Institute for Theoretical Physics, Moscow, in 2007.  Currently, he is an Associate Professor in the Mechanical Engineering Department of Massachusetts Institute of Technology (MIT), Cambridge. Before joining MIT, he held the position of Oppenheimer fellow at Los Alamos National Laboratory, and Kadanoff–Rice Postdoctoral Scholar at University of Chicago. His research interests encompass a broad range of problems involving nonlinear and stochastic dynamics of complex systems. Specific interests in energy related fields include stability and security assessment, integration of distributed and renewable generation.
\end{IEEEbiography}

\end{document}